\documentclass[a4paper,11pt,reqno]{amsart}
\usepackage{amsthm}
\usepackage{amsmath}
\usepackage{amssymb}
\usepackage{mathrsfs}
\usepackage{latexsym}
\usepackage{exscale}
\usepackage{geometry} 
\usepackage[utf8]{inputenc}
\usepackage{verbatim}
\usepackage{enumerate} 
\usepackage{ esint } 
\usepackage{fancyhdr}

\usepackage{enumitem}

\makeatletter
\@namedef{subjclassname@2020}{%
  \textup{2020} Mathematics Subject Classification}
\makeatother

\usepackage[colorlinks=true, pdfstartview=FitV, linkcolor=blue, citecolor=red, urlcolor=blue]{hyperref} 

\theoremstyle{plain}

\newtheorem*{teo*}{Theorem}
\newtheorem*{propo*}{Proposition}
\newtheorem*{lema*}{Lemma}
\numberwithin{equation}{section}
\newtheorem{teo}{Theorem}[section]
\newtheorem{lema}[teo]{Lemma}
\newtheorem{coro}[teo]{Corollary}
\newtheorem{propo}[teo]{Proposition}

\theoremstyle{definition}
\newtheorem{remark}[teo]{Remark}

\newtheorem*{defi*}{Definition}
\newtheorem{defi}[teo]{Definition}

\calclayout

\allowdisplaybreaks

\newcommand{\R}{{\mathbb{R}^n}}

\begin{document}

\baselineskip=17pt

\title[Weights for variable fractional maximal and rough operators]{Characterization of weights for the variable fractional maximal operator and weighted inequalities for variable fractional rough operators}

\author[R.~M.~Pastrana]{Rodrigo M. Pastrana}
\address{R.~M.~Pastrana\\ FaMAF \\ Universidad Nacional de C\'ordoba \\
CIEM (CONICET) \\ 5000 C\'ordoba, Argentina}
\email{rodrigo.pastrana@unc.edu.ar}

\author[M.~S.~Riveros]{M. Silvina Riveros}
\address{M.~S.~Riveros \\ FaMAF \\ Universidad Nacional de C\'ordoba \\
CIEM (CONICET) \\ 5000 C\'ordoba, Argentina}
\email{msriveros@unc.edu.ar}

\author[R.~E.~Vidal]{Ra\'ul E. Vidal}
\address{R.~E.~Vidal \\ FaMAF \\ Universidad Nacional de C\'ordoba \\
	CIEM (CONICET) \\ 5000 C\'ordoba, Argentina}
\email{raul.vidal@unc.edu.ar}


\begin{abstract} 
	We characterize the class of weights related to the boundedness of variable fractional maximal operator $M_{\beta(\cdot),r(\cdot)}$	
	on variable Lebesgue spaces. This extend previously known results, including those corresponding to the fractional operator 
	$M_{\beta(\cdot),1}$. 
	In addition, we introduce a class of kernels $K$
	satisfying a new variable Hörmander-type condition $H_{\beta(\cdot),r(\cdot)}$. For the fractional operator $T_{\beta(\cdot)}$ given by a kernel in $H_{\beta(\cdot),r(\cdot)}$, we prove a Coifman-Fefferman inequality and weighted inequalities in variable Lebesgue space. Finally, we provide examples of kernels in this variable Hörmander class.
\end{abstract}

\thanks{ The authors are  partially supported by CONICET and SECYT-UNC}



\maketitle

\section{Introduction and main results}

 Let $r(\cdot),\beta(\cdot)\in \mathcal{P}(\R)$. For a function $f\in L^{r(\cdot)}_{\text{loc}}(\R)$, the fractional maximal operator $M_{\beta(\cdot),r(\cdot)}$ is defined by
\begin{equation}\label{def M}
	M_{\beta(\cdot),r(\cdot)}f(x)=\sup_{Q\ni x}\|\chi_Q\|_{\beta(\cdot)}\frac{\|f\chi_Q\|_{r(\cdot)}}{\|\chi_Q\|_{r(\cdot)}}.
\end{equation}
where the supreme is taken over all cubes $Q$ with sides parallel to the coordinate axes.

When $\beta(\cdot)\equiv\infty$ we simply write $M_{r(\cdot)}$ instead of $M_{\infty,r(\cdot)}$. If $\beta(\cdot)\equiv\infty$ and $r(\cdot)\equiv 1$, we obtain the classical Hardy--Littlewood maximal operator, $M$. If $\beta(\cdot)\equiv n/\alpha$ with $0\leq\alpha<n$, then $M_{\beta(\cdot),1}$ reduces to the fractional maximal operator $M_{n/\alpha,1}$, usually denoted by $M_\alpha$ in the literature. If $r(\cdot)\equiv r\geq 1$ then the operator $M_{r(\cdot)}$ reduces to the maximal operator $M_r$, defined by $M_r(f)=\left(M(|f|^r)\right)^{1/r}$. If $\beta(\cdot)\equiv n/\alpha$ with $0\leq\alpha<n$ and $r(\cdot)\equiv r\geq 1$, we obtain the maximal operator $M_{n/\alpha,r}$, defined by
$$M_{n/\alpha,r}(f)(x)=\sup_{Q\ni x}|Q|^{\alpha/n}\left(\frac{1}{|Q|}\int_{Q}|f|^r\right)^{1/r},$$
where the supremum is taken over all cubes $Q$ with sides parallel to the coordinate axes.

The variable fractional maximal operator $M_{r(\cdot)}$ was introduced by Diening, Harjulehto, Hästö, and Růžička in \cite{libro Diening}. The general case of $M_{\beta(\cdot),r(\cdot)}$ was later defined by Melchiori and Pradolini in \cite{Luciana Potential}.

In \cite{libro Diening} and \cite{Luciana Potential}, the boundedness of the maximal operators $M_{r(\cdot)}$ and $M_{\beta(\cdot),r(\cdot)}$, respectively, was proved in variable Lebesgue spaces for exponents $p(\cdot)\in\mathcal{P}^{\text{log}}(\R)$ (see the definition in Section \ref{sec 2}). 
\begin{teo}[\cite{libro Diening}]\label{acot Mr}
	Let $p(\cdot),r(\cdot),s(\cdot)\in\mathcal{P}^{\text{log}}(\R)$ such that $p(\cdot)=r(\cdot)s(\cdot)$ and $s^->1$. Then 
	$$M_{r(\cdot)}:L^{p(\cdot)}(\R)\hookrightarrow L^{p(\cdot)}(\R).$$
\end{teo}
And for the general case,
\begin{teo}[\cite{Luciana Potential}]\label{acot MBr}
	Let $p(\cdot),q(\cdot),r(\cdot),s(\cdot)\in\mathcal{P}^{\text{log}}(\R)$ such that ${p(\cdot)}\leq{q(\cdot)}\leq q^+<\infty$, ${p(\cdot)}={r(\cdot)}{s(\cdot)}$ and $s^->1$. If ${\beta(\cdot)}$ is the exponent defined by $1/p(\cdot)-1/q(\cdot)=1/\beta(\cdot)$, then
	$$M_{{\beta(\cdot)},{r(\cdot)}}:L^{p(\cdot)}(\R)\hookrightarrow L^{q(\cdot)}(\R).$$
\end{teo}
Note that these works did not consider weighted estimates. 

By a weight $\omega$ we mean a locally integrable non-negative function. The classical Muckenhoupt class $A_p$ consists of all weights $\omega$ such that
$$\sup_Q\left(\frac{1}{|Q|}\int_Q\omega\right)\left(\frac{1}{|Q|}\int_Q \omega^{-\frac{1}{p-1}}\right)^{p-1}<\infty,$$
where the supremum is taken over all cubes $Q$ with sides parallel to the coordinate axes. The class $A_\infty$ is defined as $A_\infty=\bigcup_{p\geq1}A_p$.
In the variable context it is said that $\omega\in \mathcal{A}_{p(\cdot)}$ if there exists a constant $C$ such that
\begin{equation*}
	\|\omega\chi_Q\|_{p(\cdot)}\|\omega^{-1}\chi_Q\|_{p^\prime(\cdot)}\leq C|Q|
\end{equation*}
uniformly for all cubes $Q$. Note that $\omega\in \mathcal{A}_p$ if and only if $\omega^p\in A_p$.
\begin{remark}
	We use the notation $\mathcal{A}$ or $\mathbb{A}$ (defined below) for the variable classes of weights and $A$ for the classical Muckenhoupt class for constant exponents (see \cite{pesos Apq constantes,pesos Ap constantes} for the properties of $A_p$).
\end{remark}

The boundedness of the maximal operator $M$ on weighted variable Lebesgue spaces was studied in \cite{C-U Diening Hasto} by Cruz-Uribe, Diening, and Hästö, and in \cite{C-U Fiorenza Neugebauer} by Cruz-Uribe, Fiorenza, and Neugebauer. They proved that if $p(\cdot)\in \mathcal{P}^{\text{log}}(\R)$ with $1<p^-\leq p^+<\infty$, then there exists a constant $C$ such that $$\|(Mf)\omega\|_{p(\cdot)}\leq C\|f\omega\|_{p(\cdot)}$$  
if and only if $\omega\in \mathcal{A}_{p(\cdot)}$.

In the proof carried out in \cite{C-U Diening Hasto}, the boundedness of the maximal operator $M_{r(\cdot)}$ was used as a key tool. We will follow the same approach, using the boundedness of the maximal operator $M_{\beta(\cdot),r(\cdot)}$ to establish the boundedness of $M_{\beta(\cdot),r(\cdot)}$ in weighted variable Lebesgue spaces.

If $0\leq \alpha<n$ and $r(\cdot)\equiv r$, Bernardis, Dalmasso, and Pradolini defined in \cite{pesos Apq 2014} the class of weights $\mathcal{A}_{p(\cdot),q(\cdot)}$ for exponents $p(\cdot),q(\cdot)$ satisfying $1/q(\cdot)=1/p(\cdot)-\alpha/n$ with $p^+<n/\alpha$ as: $\omega\in \mathcal{A}_{p(\cdot),q(\cdot)}$ if there exists a constant $C$ such that
$$\|\chi_Q\omega\|_{q(\cdot)}\|\chi_Q\omega^{-1}\|_{p^\prime(\cdot)}\leq C|Q|^{1-\frac{\alpha}{n}}$$
uniformly for all cubes $Q$. They showed that this class characterizes the boundedness of the maximal operator $M_{n/\alpha,r}$ on variable Lebesgue spaces:
\begin{teo}[\cite{pesos Apq 2014}]
	Let $0 \leq \alpha < n$, let $w$ be a weight, and let $p(\cdot)\in\mathcal{P}^{\text{log}}(\R)$ be such that 
	$1 < p^- \le p^+ < n/\alpha$ and
	$1/q(\cdot) = 1/p(\cdot)-\alpha/n$.	Let $r$ such that 
	$1 \leq r < p^-$. 
	Then
	$$\|(M_{n/\alpha,r}f)\omega\|_{q(\cdot)}\leq C\|f\omega\|_{p(\cdot)}$$
	if and only if 
	$w^{r} \in \mathcal{A}_{\frac{p(\cdot)}{r}, \frac{q(\cdot)}{r}}$.
\end{teo}

In \cite{pesos Apq 2019}, Melchori, Pradolini, and Ramos generalized this class without restrictions on the exponents. We say that $\omega \in \mathcal{A}_{p(\cdot),q(\cdot)}$ if 
$$\|\chi_Q \omega\|_{q(\cdot)}\|\chi_Q \omega^{-1}\|_{p^\prime(\cdot)}\lesssim\|\chi_Q \|_{q(\cdot)}\|\chi_Q \|_{p^\prime(\cdot)}$$
uniformly for all cubes $Q$. They proved that for exponents $p(\cdot),q(\cdot),\beta(\cdot)$ satisfying $1/q(\cdot)=1/p(\cdot)-1/\beta(\cdot)$, if $\omega\in\mathcal{A}_{p(\cdot),q(\cdot)}$ then there exists a constant $C$ such that 
\begin{equation*}
	\|(T_{\beta(\cdot),1,\mathcal{S}}f)\omega\|_{q(\cdot)}\leq C\|f\omega\|_{p(\cdot)},	
\end{equation*}

where $\mathcal{S}$ is a sparse family (see the definition in Section \ref{sec 3}) and $T_{\beta(\cdot),1,\mathcal{S}}$ is the sparse operator defined by
$$T_{\beta(\cdot),1,\mathcal{S}}f=\sum_{Q\in \mathcal{S}}\|\chi_Q\|_{\beta(\cdot)}\|f\chi_Q\|_1 \|\chi_Q\|^{-1}_1.$$

In this work we introduce a new class of weights,  $\mathbb{A}_{p(\cdot),q(\cdot)}$, that generalizes the class $\mathcal{A}_{p(\cdot)}$ and coincides with the class $\mathcal{A}_{p(\cdot),q(\cdot)}$ when the exponents belong to $\mathcal{P}^{\text{log}}(\R)$ (see Section \ref{sec 2}). Under appropriate conditions on the exponents, these weights characterize the boundedness of the fractional maximal operator $M_{\beta(\cdot),r(\cdot)}$. This will be established below.
\begin{defi}\label{Apr}
	Let $p(\cdot),r(\cdot)\in\mathcal{P}(\R)$, such that $r(\cdot)\leq p(\cdot)$, and let $q(\cdot)$ be defined by $1/r(\cdot)=1/p(\cdot)+1/q(\cdot)$. We say that a weight $\omega$ belongs to the class $\mathbb{A}_{p{(\cdot)},r(\cdot)}$ if there exists a constant $C>0$ such that 
	$$\|\chi_Q\omega\|_{p(\cdot)}\|\chi_Q\omega^{-1}\|_{q(\cdot)}\leq C\|\chi_Q\|_{r(\cdot)}$$	 
	for all cubes $Q$ with sides parallel to the coordinate axes.
\end{defi}
\begin{teo}\label{teo M}
	Let $p(\cdot),q(\cdot),r(\cdot),s(\cdot)\in\mathcal{P}^{\log}(\R)$ such that $p(\cdot)\leq q(\cdot)\leq q^+<\infty$ and $p(\cdot)=r(\cdot)s(\cdot)$ with $s^->1$. Let $\beta(\cdot)$ be the exponent defined by $1/p(\cdot)-1/q(\cdot)=1/\beta(\cdot)$. Then exists a constant $C>0$ such that
	$$\left\|(M_{\beta(\cdot),r(\cdot)}f)\omega\right\|_{q(\cdot)}\leq C\|f\omega\|_{p(\cdot)}$$
	if and only if
	$\omega\in \mathbb{A}_{q(\cdot),\frac{\beta(\cdot) r(\cdot)}{\beta(\cdot)-r(\cdot)}}$.
\end{teo}
\begin{remark}
	 Observe that we are considering possible the case $\beta(\cdot)=\infty$.  We use the convention $\frac{r(\cdot)}{\infty}=0$ when $r(\cdot)$ is finite. Note that the expression $\frac{\beta(\cdot) r(\cdot)}{\beta(\cdot)-r(\cdot)}$ can be rewritten as $\frac{r(\cdot)}{1-\frac{r(\cdot)}{\beta(\cdot)}}$. Throughout the rest of the paper we assume this convention.
\end{remark}
If $p(\cdot)=q(\cdot)$, or equivalently $\beta(\cdot)\equiv\infty$, then we obtain the following corollary:
\begin{coro}\label{coro Mr}
	Let ${p(\cdot)},{r(\cdot)},s(\cdot)\in\mathcal{P}^{\text{log}}(\R)$ such that $p^+<\infty$, ${p(\cdot)}={r(\cdot)}{s(\cdot)}$ with $s^->1$. Then there exists a constant $C>0$ such that 
	$$\left\|(M_{r(\cdot)} f)\omega\right\|_{p(\cdot)}\leq C\|f\omega\|_{p(\cdot)}$$
	if and only if
	$\omega\in \mathbb{A}_{{p(\cdot)},{r(\cdot)}}$. 
\end{coro}
For the case $r(\cdot)\equiv1$ we obtain the following corollary:
\begin{coro}
	Let ${p(\cdot)},{q(\cdot)}\in\mathcal{P}^{\text{log}}(\R)$ such that $1<p^-\leq{p(\cdot)}\leq{q(\cdot)}\leq q^+<\infty$ and ${\beta(\cdot)}$ be the exponent defined by $1/p(\cdot)-1/q(\cdot)=1/\beta(\cdot)$. Then there exists a constant $C>0$ such that 
	$$\left\|(M_{{\beta(\cdot)},1}f)\omega\right\|_{q(\cdot)}\leq C\|f\omega\|_{p(\cdot)}$$
	if and only if
	$\omega\in \mathbb{A}_{{q(\cdot)},\beta^\prime(\cdot)}$.
\end{coro}

Let $T$ be an integral operator of the type
\begin{equation}\label{T}
	Tf(x)=\int_{\R}K(x,y)f(y)dy
\end{equation}
where $K\in L^1_{\text{loc}}\left(\mathbb{R}^{2n} \setminus \{(x,y)\in\mathbb{R}^{2n} : x=y\}\right)$.

Let $K$ be a kernel of the form $K(x,y)=\bar{K}(x-y)$, and $T$ be bounded on some $L^{p_0}(\R)$. Kurtz and Wheeden implicitly proved in \cite{Hr} that if $K$ satisfies the smoothness condition $H_r$ with $1<r<\infty$ (see the precise definition in \cite{Hr es sharp} or Definition \ref{def Hör} with $r(\cdot)\equiv r_0$ and $\beta(\cdot)\equiv\infty$), then $T$ is controlled by the maximal operator $M_{r^\prime}$. This is
\begin{equation}
\int_{\R}|Tf|^p \omega\leq C\int_{\R}\left(M_{r^\prime}f\right)^p \omega
\end{equation}
for $0<p<\infty$ and $\omega\in A_\infty$, whenever the left-side is finite.
 
In \cite{Hr es sharp}, Martell, Pérez and Trujillo-González showed that this estimate is sharp in the sense that, for $s<r^\prime$, there exists a kernel $K\in H_r$ such that the operator induced by $K$ is not controlled by $M_s$ as in (1.3). 

In \cite{H_A}, Lorente, Riveros and de la Torre proved that if $K$ satisfies the smoothness condition $H_\mathscr{A}$, then $T$ is controlled by the maximal operator $M_{\bar{\mathscr{A}}}$, where $A$ is a convex Young function, $\bar{\mathscr{A}}$ denotes its complementary function, and $M_\mathscr{A}$ is defined through the Luxemburg norm $\|\cdot\|_\mathscr{A}$ in the corresponding Orlicz space (see \cite{Oneil}). When $\mathscr{A}(t)=t^r$, the condition $H_\mathscr{A}$ reduces to the condition $H_r$. In \cite{S_A}, Bernardis, Lorente and Riveros replaced the boundedness on $L^{p_0}(\R)$ hypothesis by a size condition $\mathcal{S}_{\alpha, \mathscr{A}}$ and the condition $H_\mathscr{A}$ by $H_{\alpha,\mathscr{A}}$, under which the operator $T$ is controlled by the fractional maximal operator $M_{n/\alpha,\bar{\mathscr{A}}}$. Following these ideas we introduce a variable Hörmander-type condition and a variable size condition.
\begin{defi}\label{def Hör}
	Let $\beta(\cdot),r(\cdot)\in\mathcal{P}(\R)$, we say that a kernel $$K\in L^1_{\text{loc}}\left(\mathbb{R}^{2n} \setminus \{(x,y)\in\mathbb{R}^{2n} : x=y\}\right)$$ satisfy the
	$L^{\beta(\cdot),r(\cdot)}$-Hörmander-1 condition, and we denote $K \in
	H_{\beta(\cdot),r(\cdot),1}$, if 
	\begin{equation*} 
	\underset{Q}{\sup}\underset{x,z\in \frac{1}{2}Q}{\sup}\sum_{m=1}^{\infty} \frac{(2^m\ell(Q))^{n}}{\|\chi_{2^m Q}\|_{\beta(\cdot)}} \frac{\left\| \left[K(x,\cdot) - K(z,\cdot)\right]\chi_{2^mQ\setminus2^{m-1}Q}(\cdot)\right\|_{r(\cdot)}}{\left\|\chi_{2^m Q}\right\|_{r(\cdot)}} <\infty. 
	\end{equation*}
	Analogously, we say that $K$ satisfy the
	$L^{\beta(\cdot),r(\cdot)}$-Hörmander-2 condition, and we denote $K \in
	H_{\beta(\cdot),r(\cdot),2}$, if
	\begin{equation*} 
	\underset{Q}{\sup}\underset{x,z\in \frac{1}{2}Q}{\sup}\sum_{m=1}^{\infty} \frac{(2^m\ell(Q))^{n}}{\|\chi_{2^m Q}\|_{\beta(\cdot)}} \frac{\left\| \left[K(\cdot,x) - K(\cdot,z)\right]\chi_{2^mQ\setminus2^{m-1}Q}(\cdot)\right\|_{r(\cdot)}}{\left\|\chi_{2^m Q}\right\|_{r(\cdot)}} <\infty.
	\end{equation*}
	where supremum is taken over all cubes $Q$ with sides parallel to the coordinate axes. 
	
	We define $H_{\beta(\cdot),r(\cdot)}=H_{\beta(\cdot),r(\cdot),1}\cap H_{\beta(\cdot),r(\cdot),2}$. If $\beta(\cdot)\equiv\infty$, we simply write $H_{r(\cdot),i}$ instead of $H_{\infty,r(\cdot),i}$, and $H_{r(\cdot)}$ instead of $H_{\infty,r(\cdot)}$. If $K\in H_{r(\cdot),i}$, we say that $K$ satisfies the $L^{r(\cdot)}$-Hörmander-$i$ condition. If $K\in H_{r(\cdot)}$, we say that $K$ satisfies the $L^{r(\cdot)}$-Hörmander condition.  
\end{defi}
\begin{defi}\label{def SrB}
	Let ${\beta(\cdot)},{r(\cdot)}\in\mathcal{P}(\R)$. We say that the kernel $K$ satisfy the size condition $S_{\beta(\cdot),r(\cdot),1}$, and we write $K\in S_{\beta(\cdot),r(\cdot),1}$, if there exists a constant $C$ such that 
	\begin{equation*}
	\sup_{x\in\frac{1}{2}Q}\frac{\left\|k(x,\cdot)\chi_{2Q\setminus Q}(\cdot)\right\|_{r(\cdot)}}{\|\chi_{Q}\|_{r(\cdot)}}\leq C\frac{\|\chi_{Q}\|_{\beta(\cdot)}}{|Q|}
	\end{equation*}  
	uniformly for all cubes $Q$ with sides parallel to the coordinate axes. 
	
	Analogously, we say that the kernel $K$ satisfy the size condition $S_{\beta(\cdot),r(\cdot),2}$, and we write $K\in S_{\beta(\cdot),r(\cdot),2}$, if there exists a constant $C$ such that 
	\begin{equation*}
	\sup_{y\in\frac{1}{2}Q}\frac{\left\|k(\cdot,y)\chi_{2Q\setminus Q}(\cdot)\right\|_{r(\cdot)}}{\|\chi_{Q}\|_{r(\cdot)}}\leq C\frac{\|\chi_{Q}\|_{\beta(\cdot)}}{|Q|}
	\end{equation*}  
	uniformly for all cubes $Q$ with sides parallel to the coordinate axes. We define $S_{\beta(\cdot),r(\cdot)}=S_{\beta(\cdot),r(\cdot),1}\cap S_{\beta(\cdot),r(\cdot),2}$.
\end{defi}

 When $\mathscr{A}(t)=t^{r_0}$, the classes $S_{\beta(\cdot),r_0}$ are a generalization of the classes $\mathcal{S}_{\alpha,\mathscr{A}}$ defined in \cite{S_A} in the context of variable Lebesgue spaces.

The classes $H_{\beta(\cdot),r(\cdot),i}$ and $S_{\beta(\cdot),r(\cdot),i}$ are decreasing in the exponent \\$r(\cdot)\in\mathcal{P}^{\text{log}}(\R)$. Moreover, the inclusions
\begin{equation}\label{contenciones de Hr}
	H_{\infty,i}\subset H_{q(\cdot),i}\subset H_{p(\cdot),i}\subset\cdots \subset H_{1,i}
\end{equation}

 are strict, as we will see with the examples of kernels presented in Section~\ref{sec 4}.

Throughout, $T$ denotes an operator with kernel $K\in H_{r(\cdot)}$. When the kernel $K$ belongs to $H_{\beta(\cdot),r(\cdot),1} \cap S_{\beta(\cdot),r(\cdot),2}$, we write $T_{\beta(\cdot)}$ instead of $T$.

We prove that an integral operator of the type \eqref{T} satisfying these conditions is controlled by the maximal operator $M_{\beta(\cdot),r(\cdot)}$. In other words, we prove a Coifman-Fefferman inequality for these operators:

\begin{teo}\label{teo Hör}
	Let $r(\cdot)\in\mathcal{P}^{\log}(\R)$, and let $T$ be an integral operator bounded on some $L^{p_0}(\R)$ with $1<p_0<\infty$, whose kernel $K$ belongs to $H_{r(\cdot)}$. Then
	\begin{enumerate}[label=\alph*)]
		\item \label{a del teo Hör}
		For any $p$ with $0<p<\infty$, and any $w\in A_\infty$, there exists a constant $C$ such that
		\begin{equation*}
			\int_{\R}|Tf|^p w\leq C\int_{\R}\left(M_{r^\prime(\cdot)}f\right)^p w
		\end{equation*}
		whenever the left-hand side is finite.	
		\item \label{b del teo Hör}
		If $p(\cdot)\in\mathcal{P}^{\text{log}}(\R)$ with $1<p^-\leq p^+<\infty$, and $\omega\in \mathcal{A}_{p(\cdot)}$ then there exists a constant C such that
		\begin{equation*}
			\|(Tf)\omega\|_{p(\cdot)}\leq C \|(M_{r^\prime(\cdot)}f)\omega\|_{p(\cdot)}
		\end{equation*}
		whenever the left-hand side is finite.	
	\end{enumerate}
\end{teo}
As a consequence of the previous theorem, we have the following weighted strong estimate for the integral operator $T$.
\begin{teo}\label{teo T}
		Let ${p(\cdot)},{r(\cdot)},s(\cdot)\in\mathcal{P}^{\log}(\R)$ such that $p^+<\infty$ and ${p(\cdot)}={r^\prime(\cdot)}{s(\cdot)}$ with $s^->1$. Let $T$ be an integral operator bounded on some $L^{p_0}(\R)$ with $1<p_0<\infty$, whose kernel $K$ belongs to $H_{r(\cdot)}$. If $\omega\in \mathbb{A}_{{p(\cdot)},{r^\prime(\cdot)}}$ then there exists a constant $C>0$ such that 
	$$\left\|(T f)\omega\right\|_{p(\cdot)}\leq C\|f\omega\|_{p(\cdot)}$$
	   whenever the left-hand side is finite.	
 \end{teo}

\begin{teo}\label{teo Hör frac}
	Let $r(\cdot),\beta(\cdot)\in\mathcal{P}^{\log}(\R)$ such that $\beta^+<\infty$, and let $T_{\beta(\cdot)}$ be an integral operator whose kernel $K$ belongs to $H_{\beta(\cdot),r(\cdot),1}\cap S_{\beta(\cdot),1,2}$.
	\begin{enumerate}[label=\alph*)]
		\item For any $p$ with $1<p<\infty$ and any $w\in A_\infty$, there exists a constant $C$ such that
		\begin{equation*}
			\int_{\R}|T_{\beta(\cdot)}f|^p w\leq 	C\int_{\R}\left(M_{\beta(\cdot),r^\prime(\cdot)}f\right)^p w
		\end{equation*}
	whenever the left-hand side is finite.
		\item 
		
		If $p(\cdot)\in\mathcal{P}^{\text{log}}(\R)$ with $1<p^-\leq p^+<\infty$, and $\omega\in \mathcal{A}_{p(\cdot)}$ then there exists a constant C such that
	\begin{equation*}
		\|\left(T_{\beta(\cdot)}f\right)\omega\|_{p(\cdot)}\leq C \|(M_{\beta(\cdot),r(\cdot)})\omega\|_{p(\cdot)}
	\end{equation*}
	whenever the left-hand side is finite.	
	\end{enumerate}
\end{teo}
In the fractional case we also obtain the following weighted strong estimate for the operator $T_{\beta(\cdot)}$.
\begin{teo}\label{teo TB}
	Let $p(\cdot),q(\cdot),r(\cdot),s(\cdot)\in\mathcal{P}^{\log}(\R)$ such that $p(\cdot)<q(\cdot)\leq q^+<\infty$ and $p(\cdot)=r^\prime(\cdot)s(\cdot)$ with $s^->1$. Let $\beta(\cdot)$ be the exponent defined by $1/\beta(\cdot)=1/p(\cdot)-1/q(\cdot)$. Let $T_{\beta(\cdot)}$ be an integral operator whose kernel $K$ belongs to $H_{\beta(\cdot),r(\cdot),1}\cap S_{\beta(\cdot),1,2}$. If $\beta^+<\infty$ and $\omega\in \mathbb{A}_{q(\cdot),\frac{\beta(\cdot) r^\prime(\cdot)}{\beta(\cdot)-r^\prime(\cdot)}}$, then exists a constant $C$ such that
	$$\left\|\left(T_{\beta(\cdot)}f\right)\omega\right\|_{q(\cdot)}\leq C\|f\omega\|_{p(\cdot)}$$
	whenever the left-hand side is finite.	
\end{teo}
\begin{remark}
	The examples presented in Section \ref{sec 4} satisfy that the left-hand sides of the inequalities in Theorems \ref{teo T} and \ref{teo TB} are finite. See Propositions \ref{lado izq finito K}, \ref{lado izq finito K tilde} and \ref{lado izq finito K alpha}. 	
\end{remark}

The sharp maximal function $M^\sharp$ is defined as
\begin{equation*}
M^\sharp f(x)=\sup_{Q\ni x}\inf_{a\in\mathbb{R}}\frac{1}{|Q|}\int_Q |f(y)-a|dy.
\end{equation*}
The weighted Bounded Mean Oscillation space, $BMO_\omega$ is defined by
\begin{equation*}
	BMO_\omega=\{f\in L^1_{\text{loc}}(\R): \omega M^\sharp f\in L^{\infty}(\R) \},
\end{equation*}
The seminorm on this space is given by
\begin{equation*}
\|f\|_{*,\omega}=\sup_Q \|\chi_Q\omega\|_\infty\left[\frac{1}{|Q|}\int_Q\left(f(x)-\frac{1}{|Q|}\int_Qf(y)dy\right)dx\right].
\end{equation*}
Is easy to see that $\|f\|_{*,\omega}\approx\left\|\omega M^\sharp f\right\|_\infty$ (For details, see \cite{BMOw}.)
\begin{teo}\label{borde}
	Let ${r(\cdot)},{\beta(\cdot)}\in\mathcal{P}^{\text{log}}(\R)$ such that $r^\prime(\cdot)\leq\beta(\cdot)\leq \beta^+<\infty$. Let $T_{\beta(\cdot)}$ an integral operator whose kernel $k$ belongs to $H_{{\beta(\cdot)},{r(\cdot)},1}\cap S_{{\beta(\cdot)},1,2}$. If $\omega\in \mathbb{A}_{\infty,\frac{{r^\prime(\cdot)}{\beta(\cdot)}}{{\beta(\cdot)}-{r^\prime(\cdot)}}}$, then there exists a constant $C$ such that
	\begin{equation*}
	\|T_{\beta(\cdot)}f\|_{*,\omega}\leq C \|\omega f\|_{\beta(\cdot)}.
	\end{equation*}
\end{teo}

The rest of the paper is organized as follows:
in Section \ref{sec 2} we state some necessary results concerning the variable Lebesgue spaces and we give some properties of the class $\mathbb{A}_{p(\cdot),r(\cdot)}$. In Section \ref{sec 3} we prove Theorem \ref{teo M}. In Section \ref{sec 4} we prove Theorems \ref{teo Hör}, \ref{teo T}, \ref{teo Hör frac}, \ref{teo TB}, and \ref{borde}. We also give different examples of operators whose kernels satisfy the conditions \ref{def Hör} and \ref{def SrB}. In Section \ref{sec 5}, we present some lemmas and technical proofs of some results that are necessary but not central to the paper. In the rest of the paper, $C$ will denote constants that may vary from line to line. When we write $f\lesssim g$, we mean that  $f\leq C g$, and when we write $f\approx g$, we mean that both $f\lesssim g$ and $g\lesssim f$ hold. We will denote by $Q$ a cube with sides parallel to the coordinate axes., $c_Q$ its center, and $\ell(Q)$ its side length.
 For $m\in(0,\infty)$, $mQ$ denotes the cube concentric with $Q$ and $\ell(mQ)=m\ell(Q)$.

\section{Preliminaries}\label{sec 2}
\subsection{Variable Lebesgue space}
\

Next, we give some definitions and state some basic results on variable Lebesgue spaces.

We denote by $\mathcal{P}(\mathbb{R})$ the set of exponents consisting of all measurable functions $p(\cdot): \mathbb{R} \to [1,\infty]$. For $p(\cdot) \in \mathcal{P}(\mathbb{R})$, we single out the following sets:
 \begin{align*}
 (\R)^{p(\cdot)}_\infty&=\left\{x\in\R : p(x)=\infty\right\}\\
 (\R)^{p(\cdot)}_1&=\left\{x\in\R : p(x)=1\right\}\\
 (\R)^{p(\cdot)}_*&=\left\{x\in\R : 1<p(x)<\infty\right\}.
 \end{align*} 
 For a measurable $E\subset\R$, denote by $p^+_E=\text{esssup}_E\ p(\cdot)$ and $p^-=\text{essinf}_E\ p(\cdot)$. If $E=\R$ we simply write $p^+$ and $p^-$ respectively. The conjugated exponent $p^\prime(\cdot)$ is defined as usually, by, 
 $1/p(\cdot)+1/p^\prime(\cdot)=1$.
  The variable Lebesgue space $L^{p(\cdot)}(\R)$ is defined as the set of all measurable functions on $\R$ such that, for some $\lambda>0$, the modular convex function satisfy
\begin{equation*}
	\rho_{p(\cdot)}\left(\frac{f}{\lambda}\right)=\int_{\R\setminus(\R)^{p(\cdot)}_\infty}\left|\frac{f(x)}{\lambda}\right|^{p(x)}dx+\left\|f\chi_{(\R)^{p(\cdot)}_\infty}\right\|_\infty<\infty.
\end{equation*}
This is a Banach space with norm
\begin{equation}\label{def norma}
	\|f\|_{p(\cdot)}=\inf\left\{\lambda>0:	\rho_{p(\cdot)}(f/\lambda)\leq 1\right\},
\end{equation}
and generalized the clasical Lebesgue space, if $p(\cdot)$ is constantly $p_0$ then $L^{p(\cdot)}(\R)=L^{p_0}(\R)$ and $\|\cdot\|_{p(\cdot)}=\|\cdot\|_{p_0}$.

Given a weight $\omega$ (a non-negative measurable function), the weighted variable Lebesgue space, $L^{p(\cdot)}(\omega)$, is defined as the set of all measurable functions $f$ such that $f\omega\in L^{p(\cdot)}(\R)$. The norm in this space is given by $\|f\omega\|_{p(\cdot)}$.
\begin{lema}[\cite{Extrapol con pesos}]\label{densidad en Lp de omega}
	Given $p(\cdot)\in \mathcal{P}(\R)$ with $p^+<\infty$, and a weight $\omega\in L_{\text{loc}}^{p(\cdot)}$, then the smooth functions of compact support, $C^\infty_c$, are dense in $L^{p(\cdot)}(\omega)$.
\end{lema}

The following properties hold:
\begin{propo}[\cite{libro Diening,libro CU-FIO}]\label{propo modular} Let $p(\cdot)\in \mathcal{P}(\R)$:
	\begin{enumerate}
		\item If $|f(x)|\leq|g(x)|$, then $\rho_{p(\cdot)}(f)\leq \rho_{p(\cdot)}(g)$ and $\|f\|_{p(\cdot)}\leq\|g\|_{p(\cdot)}$.
		\item If there exist $\lambda_0>0$ such that $\rho_{p(\cdot)}(f/\lambda_0)<\infty$, then the funcition $\lambda\longmapsto\rho_{p(\cdot)}(f/\lambda)$ is continuous and drecreasing. Moreover 
		
		$\rho_{p(\cdot)}(f/\lambda)\to 0$ when $\lambda\to\infty$.
		\item If $\alpha>1$, then $\alpha\rho_{p(\cdot)}(f)\leq \rho_{p(\cdot)}(\alpha f)$. If $0<\alpha<1$, then $\rho_{p(\cdot)}(\alpha f)\leq \alpha\rho_{p(\cdot)}(f)$.
		\item If $\|f\|_{p(\cdot)}\leq 1$, then $\rho_{p(\cdot)}(f)\leq \|f\|_{p(\cdot)}$. If $\|f\|_{p(\cdot)}\geq 1$, then $\|f\|_{p(\cdot)}\leq\rho_{p(\cdot)}(f)$.
	\end{enumerate}
\end{propo}
If $p^+<\infty$, then for every constant $s_0\geq1/p^-$, it follows from the definition that
\begin{equation}\label{exponente en la norma}
	\left\||f|^{s_0}\right\|_{p(\cdot)}=\|f\|^{s_0}_{s_0p(\cdot)}.
\end{equation}

The restriction on $s_0$ can be replaced by $s_0 > 0$ if we define $\|\cdot\|_{s_0 p(\cdot)}$ as \eqref{def norma}, even when $s_0 p(\cdot) < 1$.

For this norm, we have the following generalized version of Hölder's inequality.
\begin{equation*}
	\|fg\|_1\leq 4\|f\|_{p(\cdot)}\|g\|_{p^\prime(\cdot)},
\end{equation*}
even more general: 
\begin{teo}[Hölder’s inequality, \cite{libro Diening,libro CU-FIO}]\label{teo Hölder}
		Let $r(\cdot),p(\cdot),q(\cdot)\in\mathcal{P}(\R)$ such that $1/r(\cdot)=1/p(\cdot)+1/q(\cdot)$, then
		\begin{equation*}
		\|fg\|_{r(\cdot)}\leq C\|f\|_{p(\cdot)}\|g\|_{q(\cdot)}.
		\end{equation*}
\end{teo}
We say that $p(\cdot)$ satisfies the continuous log-Hölder condition, and we write $p(\cdot)\in LH$, if there exist constants $c_0$, $c_\infty$ and $p_\infty$ such that
\begin{itemize}
    \item
	\begin{equation*}
	|p(x)-p(y)|\leq\frac{c_0}{-\log(|x-y|)},\qquad x,y \in\R,\ |x-y|<\frac{1}{2}.
	\end{equation*} 
	\item
	\begin{equation*}
	|p(x)-p_\infty|\leq\frac{c_\infty}{\log(e+|x|)}\qquad x\in\R.
	\end{equation*}
\end{itemize} 
	We denote by $\mathcal{P}^{\text{log}}(\R)$ the class of all exponents $p(\cdot)$ such that $1/p(\cdot)\in LH$. The set $\mathcal{P}^{\text{log}}(\R)$ satisfies the following properties:
	\begin{lema}[\cite{libro CU-FIO}] \label{algebra de Plog} Let $p(\cdot),q(\cdot)\in\mathcal{P}(\R)$ and let $r(\cdot)$ be the exponent defined by $1/r(\cdot)=1/p(\cdot)+1/q(\cdot)$.  Then:
		\begin{enumerate}
			\item If $p(\cdot)\in \mathcal{P}^{\text{log}}(\R)$ then $p^\prime(\cdot)\in\mathcal{P}^{\text{log}}(\R)$.
			\item If $p(\cdot),q(\cdot)\in \mathcal{P}^{\text{log}}(\R)$ then $r(\cdot)\in\mathcal{P}^{\text{log}}(\R)$.
			\item If $r(\cdot),p(\cdot)\in \mathcal{P}^{\text{log}}(\R)$ then $q(\cdot)\in\mathcal{P}^{\text{log}}(\R)$.
			\item If $p^+,q^+<\infty$ then $p(\cdot)q(\cdot)\in\mathcal{P}^{\text{log}}(\R)$.
		\end{enumerate}	
	\end{lema}
	For a cube $Q$, the harmonic mean is defined by
	$$\frac{1}{p_Q}=\frac{1}{|Q|}\int_{Q}\frac{1}{p(x)}dx.$$ 
	For exponents $p(\cdot)\in\mathcal{P}^{\text{log}}(\R)$, the following estimates for the norm of the indicator function of a cube $Q$ are know:
	\begin{lema}[\cite{libro Diening}]\label{norma de cubos comparada con |Q|}
	 	Let $p(\cdot)\in \mathcal{P}^{\text{log}}(\R)$. Then $\|\chi_Q\|_{p(\cdot)}\approx|Q|^{1/p_Q}$ for all cubes $Q$. Moreover
	 	\begin{equation*}
	 		\|\chi_Q\|_{p(\cdot)}\approx \left\{ \begin{array}{lcc} |Q|^{\frac{1}{p(x)}} & \text{if} & |Q|\leq 2^n \ \text{and} \ x\in Q,\\ \\ |Q|^{\frac{1}{p_\infty}} & \text{if} & |Q|\geq 1 \end{array} \right.
	 	\end{equation*}
	 	uniformly for all cubes $Q$.
	\end{lema}
	\begin{lema}[\cite{libro Diening}]\label{norma de cubos}
		Let $p(\cdot),q(\cdot),r(\cdot)\in \mathcal{P}^{\text{log}}(\R)$ such that $1/r(\cdot)=1/p(\cdot)+1/q(\cdot)$. Then 
		$$\|\chi_Q\|_{r(\cdot)}\approx\|\chi_Q\|_{p(\cdot)}\|\chi_Q\|_{q(\cdot)}$$
		uniformly for all cubes $Q$.
	\end{lema}
	For exponents $p(\cdot)\in\mathcal{P}^{\text{log}}(\R)$, there exists a constant $D_{p(\cdot)}$ such that for all cube $Q$.
		\begin{equation}\label{doblante}
			\|\chi_{2Q}\|_{p(\cdot)}\leq D_{{p(\cdot)}}\|\chi_Q\|_{p(\cdot)}.
		\end{equation}
	The following result extends a local property to the global setting.
	\begin{teo}[\cite{C-U Diening Hasto}]\label{teo suma de normas}
		If $p(\cdot)\in\mathcal{P}^{\text{log}}(\R)$ and $\mathcal{G}$ is a family of pairwise disjoint cubes, then 
		\begin{equation*}
			\sum_{Q\in\mathcal{G}}\|f\chi_Q\|_{p(\cdot)} \|g\chi_Q\|_{p^\prime(\cdot)}\leq C \|f\|_{p(\cdot)}\|g\|_{p^\prime(\cdot)}
		\end{equation*}
		for every pair of functions $f\in L^{p(\cdot)}_{\text{loc}}(\R)$ and $g\in L^{p^\prime(\cdot)}_{\text{loc}}(\R)$.
	\end{teo}

\subsection{Classes of weights $\mathbb{A}_{p(\cdot),r(\cdot)}$}
\

The classes $\mathcal{A}_{p(\cdot)}$ are left open in the following sense.
\begin{propo}[\cite{C-U Diening Hasto}]\label{Ap abierto izq}
	Let $p(\cdot)\in\mathcal{P}^{\text{log}}(\R)$ with $p^->1$ and $\omega\in\mathcal{A}_{p(\cdot)}$. Then there exists $s_0<1$ such that for every $s\in (s_0,1)$, $\omega^{1/s}\in\mathcal{A}_{sp(\cdot)}$.
\end{propo}
\begin{lema}[\cite{C-U Diening Hasto}]\label{def equiv de Ap}
	Let $p(\cdot)\in\mathcal{P}^{\text{log}}(\R)$ such that $1<p^-\leq p^+<\infty$. Then, $\omega\in\mathcal{A}_{p(\cdot)}$ if and only if
	$$\|\omega(\cdot)^{\cdot}\|_1\|\omega(\cdot)^{-p(\cdot)}\|_{p^\prime(\cdot)/p(\cdot)}\leq C|Q|^{p_Q}$$
	uniformly for all cube, where $\|\cdot\|_{p^\prime(\cdot)/p(\cdot)}$ is defined as \eqref{def norma} even when $p^\prime(\cdot)/p(\cdot)<1$.
	 
\end{lema}

\begin{remark}\label{equiv de Apq}
	By definition of the class $\mathbb{A}_{p(\cdot),r(\cdot)}$, if $r(\cdot)\equiv1$ then $$\mathbb{A}_{p(\cdot),1}=\mathcal{A}_{p(\cdot)}.$$ Moreover, if $p(\cdot), q(\cdot) \in \mathcal{P}^{\mathrm{log}}(\mathbb{R})$ with $p(\cdot) \le q(\cdot)$, and if $\beta(\cdot)$ is the exponent defined by $1/{p(\cdot)} - 1/{q(\cdot)} = 1/{\beta(\cdot)}$,
	then, by Lemmas \ref{algebra de Plog} and \ref{norma de cubos}, we have
	$\mathcal{A}_{p(\cdot),q(\cdot)} = \mathbb{A}_{q(\cdot),\beta'(\cdot)}$.
\end{remark}
If either $p(\cdot)$ or $q(\cdot)$ is constant, then the class $\mathbb{A}_{p(\cdot), q(\cdot)}$ reduces to some $\mathcal{A}_{t(\cdot)}$ in the following sense.
\begin{lema}
	Let ${p(\cdot)}\in\mathcal{P}(\R)$ and let $r\leq p^{-}$. Then $\omega\in\mathbb{A}_{{p(\cdot)},r}$ if and only if $\omega^r\in \mathcal{A}_{p(\cdot)/r}$. 
\end{lema}
\begin{proof}
The proof follows immediately from the definitions and \eqref{exponente en la norma}.
\end{proof}
\begin{coro}
	Let $p\geq r$. Then $\omega\in \mathbb{A}_{p,r}$ if and only if $\omega^r\in \mathcal{A}_{p/r}$ if and only if $\omega^p\in A_{p/r}$.
\end{coro}
\begin{lema}\label{Lema equivalencia entre clases Apr y Ap/r} Let ${r(\cdot)}\in \mathcal{P}^{\text{log}}(\R)$ such that $r^+<\infty$, and let $p$ such that $r^+<p<\infty$. Then
	$\omega\in\mathbb{A}_{p,{r(\cdot)}}$ if and only if $\omega^r\in \mathcal{A}_{p/r(\cdot)}$.
\end{lema}
\begin{proof} By Lemma \ref{def equiv de Ap} and \eqref{exponente en la norma}, it suffices to show that for every cube $Q$,
	$$|Q|^{\frac{p}{r_Q}} \approx \|\chi_Q\|_{r(\cdot)}^p,$$
	which is a direct consequence of Lemma \ref{norma de cubos comparada con |Q|}.
\end{proof}

\begin{lema}\label{Apr implica Ap}
	Let ${p(\cdot)},{r(\cdot)}\in\mathcal{P}^{\mathrm{log}}(\R)$ such that ${r(\cdot)}\leq{p(\cdot)}$. If $\omega\in\mathbb{A}_{{p(\cdot)},{r(\cdot)}}$ then $\omega\in\mathcal{A}_{p(\cdot)}$ and $\omega^{-1}\in \mathcal{A}_{{q(\cdot)}}$ where $1/r(\cdot)=1/p(\cdot)+1/q(\cdot)$.
\end{lema}
\begin{proof} By Lemma \ref{algebra de Plog}, $q(\cdot)\in\mathcal{P}^{\mathrm{log}}(\R)$.  Since $\omega\in \mathbb{A}_{p(\cdot),r(\cdot)}$ is equivalent to $\omega^{-1}\in\mathbb{A}_{q(\cdot),r(\cdot)}$, it is enough to show that $\omega\in\mathcal{A}_{p(\cdot)}$. Note that $1/p^\prime(\cdot)=1/q(\cdot)+1/r^\prime(\cdot).$ Then by Hölder's inequality \ref{teo Hölder}, the condition $\mathbb{A}_{p(\cdot),r(\cdot)}$ and Lemma \ref{norma de cubos}, 
	\begin{align*}
		\|\omega\chi_Q\|_{p(\cdot)}\|\omega^{-1}\chi_Q\|_{p^\prime(\cdot)}&\lesssim \|\omega\chi_Q\|_{p(\cdot)}\|\omega^{-1}\chi_Q\|_{q(\cdot)}\|\chi_Q\|_{r^\prime(\cdot)}\\
		&\lesssim\|\chi_Q\|_{r(\cdot)}\|\chi_Q\|_{r^\prime(\cdot)}
		\lesssim|Q|
	\end{align*}
	 for all cubes $Q$.
\end{proof}

\section{Characterization of weights for $M_{\beta(\cdot),r(\cdot)}$. }\label{sec 3}
In this section we characterize the weights for the maximal fractional operator $M_{\beta(\cdot),r(\cdot)}$. Before proving Theorem \ref{teo M}, we establish some auxiliary results.
	Let $p(\cdot),{r(\cdot)}\in\mathcal{P}(\R)$ such that ${r(\cdot)}\leq{p(\cdot)}$. Let ${q(\cdot)}$ be the exponent defined by	$$\frac{1}{r(\cdot)}=\frac{1}{p(\cdot)}+\frac{1}{q(\cdot)}.$$ 
	For $f\in L^1_{\text{loc}}(\R)$ we define
	\begin{equation}\label{norma conjugada}
		\|f\|^\ast_{p(\cdot)}=\underset{\|g\|_{q(\cdot)}\leq1}{\sup}\|fg\|_{r(\cdot)}.
	\end{equation}

It is not hard to see that $\|\cdot\|^\ast_{p(\cdot)}$ defines a norm on $L^{p(\cdot)}(\R)$. Moreover, under appropriate conditions on the exponents ${p(\cdot)}$ and ${r(\cdot)}$, the norms $\|\cdot\|^\ast_{p(\cdot)}$ and $\|\cdot\|_{p(\cdot)}$ are equivalent.
\begin{propo}\label{propo norma conjugada}
	Let ${p(\cdot)},{q(\cdot)},{r(\cdot)}\in\mathcal{P}(\R)$ such that 
	$$\frac{1}{r(x)}=\frac{1}{p(x)}+\frac{1}{q(x)}.$$ 
	If either $r(\cdot)<p(\cdot)<\infty$ or $r(\cdot)\leq p(\cdot)<\infty$ with $r^+<\infty$, then 
	\begin{equation*}
		\|f\|^\ast_{p(\cdot)}=\underset{\|g\|_{q(\cdot)}\leq1}{\sup}\|fg\|_{r(\cdot)}\approx\|f\|_{p(\cdot)}.
	\end{equation*}
\end{propo}

The proof can be found in Section \ref{sec 5}.
In the particular case $r(\cdot)\equiv 1$, this coincides with the conjugate norm (see \cite{libro Diening,libro CU-FIO}).

The class introduced in \ref{Apr} is closely related with the boundedness of the average operator ${\bf A}_{\beta(\cdot),r(\cdot),Q}$ defined by 
$${\bf A}_{\beta(\cdot),r(\cdot),Q}f=\chi_Q \|\chi_Q\|_{\beta(\cdot)}\frac{ \|f\chi_Q\|_{r(\cdot)}}{\|\chi_Q\|_{r(\cdot)}},$$
as shown below. 
\begin{propo}\label{caracterizacion Apr}
	Let $\omega$ be a weight. Let $p(\cdot),q(\cdot),\beta(\cdot)\in\mathcal{P}(\R)$ such that $p(\cdot)\leq q(\cdot)$ and $1/p(\cdot)-1/q(\cdot)=1/\beta(\cdot)$.
	
	\begin{enumerate}
		\item Let $r(\cdot)\in\mathcal{P}(\R)$ such that either ${r(\cdot)}<{p(\cdot)}$ or ${r(\cdot)}\leq{p(\cdot)}$ with $r^+<\infty$. If 
		$$\|({\bf A}_{{\beta(\cdot)},{r(\cdot)},Q}f)\omega\|_{q(\cdot)}\lesssim\|f\omega\|_{p(\cdot)}$$
		uniformly for all cubes $Q$, then $\omega\in \mathbb{A}_{{q(\cdot)},\frac{{\beta(\cdot)}{r(\cdot)}}{{\beta(\cdot)}-{r(\cdot)}}}$.	
		\item\label{2 de caract}  Let $r(\cdot)\in\mathcal{P}^{\text{log}}(\R)$ such that ${r(\cdot)}\leq{p(\cdot)}$ and let
		
		 $\beta(\cdot)\in\mathcal{P}^{\text{log}}(\R)$. If $\omega\in \mathbb{A}_{{q(\cdot)},\frac{{\beta(\cdot)}{r(\cdot)}}{{\beta(\cdot)}-{r(\cdot)}}}$ then 	
		$$\|({\bf A}_{{\beta(\cdot)},{r(\cdot)},Q}f)\omega\|_{q(\cdot)}\lesssim\|f\omega\|_{p(\cdot)}$$
		uniformly for all cubes $Q$.
	\end{enumerate}
\end{propo}

\begin{proof}
	Note that $$\frac{1}{{p(\cdot)}}-\frac{1}{{q(\cdot)}}=\frac{1}{{\beta(\cdot)}}\qquad \text{if and only if}\qquad
	\frac{1}{{q(\cdot)}}+\frac{{p(\cdot)}-{r(\cdot)}}{{p(\cdot)}{r(\cdot)}}=\frac{{\beta(\cdot)}-{r(\cdot)}}{{\beta(\cdot)}{r(\cdot)}}.$$
	Furthermore, ${r(\cdot)}\leq{p(\cdot)}\leq{\beta(\cdot)}$ and therefore $$\frac{1}{{r(\cdot)}}=\frac{1}{{\beta(\cdot)}}+\frac{{\beta(\cdot)}
		-{r(\cdot)}}{{\beta(\cdot)}{r(\cdot)}}\qquad\text{ and }\qquad\frac{1}{{r(\cdot)}}=\frac{1}{{p(\cdot)}}+\frac{{p(\cdot)}
		-{r(\cdot)}}{{p(\cdot)}{r(\cdot)}}.$$
	\begin{enumerate}
		\item By Proposition \ref{propo norma conjugada} applied to the norm $\left\|\cdot\right\|_{\frac{{p(\cdot)}{r(\cdot)}}{{p(\cdot)}-{r(\cdot)}}}$ and Hölder's inequality \ref{teo Hölder}, we get,
		\begin{align*}
			\|\omega\chi_Q\|_{q(\cdot)}\|\omega^{-1}\chi_Q\|_\frac{{p(\cdot)}{r(\cdot)}}{{p(\cdot)}-{r(\cdot)}}&\lesssim\|\omega\chi_Q\|_{q(\cdot)}\sup_{\|g\|_{p(\cdot)} \leq1}\|g\omega^{-1}\chi_Q\|_{r(\cdot)}\\
			&\lesssim \|\chi_Q\|_{\frac{{\beta(\cdot)}{r(\cdot)}}{{\beta(\cdot)}-{r(\cdot)}}}\sup_{\|g\|_{p(\cdot)} \leq1} \left\|\omega {\bf A}_{{\beta(\cdot)},{r(\cdot)},Q}\left(g\omega^{-1}\right)\right\|_{q(\cdot)}\\
			&\lesssim\|\chi_Q\|_{\frac{\beta(\cdot)r(\cdot)}{\beta(\cdot)-r(\cdot)}}.
		\end{align*}

		\item If $\omega\in\mathbb{A}_{{q(\cdot)},\frac{{\beta(\cdot)}{r(\cdot)}}{{\beta(\cdot)}-{r(\cdot)}}}$, by Hölder's inequality \ref{teo Hölder} and Lemmas \ref{algebra de Plog} and \ref{norma de cubos}, we obtain,
			\begin{align*}
			\left\| \left({\bf A}_{{\beta(\cdot)},{r(\cdot)},Q}f\right)\omega\right\|_{q(\cdot)}&=\frac{\|\chi_Q\|_{\beta(\cdot)}}{\|\chi_Q\|_{r(\cdot)}}\|f\chi_Q \omega\omega^{-1}\|_{r(\cdot)}\|\omega\chi_Q\|_{q(\cdot)}\\
			&\lesssim\|f\omega\|_{p(\cdot)} \frac{\|\omega\chi_Q\|_{q(\cdot)}\|\omega^{-1}\chi_Q\|_{\frac{{p(\cdot)}{r(\cdot)}}{{p(\cdot)}-{r(\cdot)}}}}{\|\chi_Q\|_{\frac{{\beta(\cdot)}{r(\cdot)}}{{\beta(\cdot)}-{r(\cdot)}}}}\lesssim \|f\omega\|_{p(\cdot)}.
		\end{align*}
	\end{enumerate}
\end{proof}
\begin{remark}
	If $\beta(\cdot)\equiv\infty$ then the hypothesis, $r(\cdot)\in\mathcal{P}^{\text{log}}(\R)$, in $\eqref{2 de caract}$ can be replaced by $r(\cdot)\in\mathcal{P}(\R)$. 
\end{remark}
\begin{coro}
	Let $p(\cdot),r(\cdot)\in \mathcal{P}(\R)$ such that either ${r(\cdot)}<{p(\cdot)}$ or ${r(\cdot)}\leq{p(\cdot)}$ with $r^+<\infty$. Then $$\|({\bf A}_{r(\cdot),Q}f)\omega\|_{p(\cdot)}\lesssim\|f\omega\|_{p(\cdot)}$$
	uniformly for all cubes $Q$ if and only if $\omega\in \mathbb{A}_{p(\cdot),r(\cdot)}$.	
\end{coro}
\begin{remark}
	Observe that if we consider weights of the type $\mathcal{A}_{p(\cdot),q(\cdot)}$ in the proof of Proposition \ref{caracterizacion Apr}, we also need the exponents to lie in $\mathcal{P}^{\text{log}}(\R)$, which restricts the class of admissible exponents.
\end{remark}

We also need a version of the Calderón-Zygmund decomposition. Let $\mathcal{D}$ be a dyadic collection of cubes. We say that a family $\mathcal{S}\subset\mathcal{D}$ is $\eta$-sparse, with $\eta\in(0,1)$, if for each $Q\in \mathcal{S}$ there exists a measurable set $E_Q\subset Q$ such that 
	$$\eta|Q|\leq |E_Q|$$
	and the sets $E_Q$ are pairwise disjoint. 
	
 We define the dyadic maximal operator $M^{\mathcal{D}}_{\beta(\cdot),r(\cdot)}$ by
	\begin{equation*}
		M^{\mathcal{D}}_{\beta(\cdot),r(\cdot)}f(x)=\sup_{\substack{Q\ni x\\ Q\in\mathcal{D}}}\|\chi_Q\|_{\beta(\cdot)}\frac{\|f\chi_Q\|_{r(\cdot)}}{\|\chi_Q\|_{r(\cdot)}}.
	\end{equation*}

\begin{lema}\label{lema descomposición de C-Z}
	Let ${\beta(\cdot)},{r(\cdot)}\in\mathcal{P}^{\text{log}}(\R)$ and let $f$ a measurable function. If
	\begin{equation}\label{hipotesis de C-Z}
		\|\chi_Q\|_{\beta(\cdot)}\frac{\|f\chi_Q\|_{r(\cdot)}}{\|\chi_Q\|_{r(\cdot)}}\underset{|Q|\to\infty}{\longrightarrow}0,
	\end{equation}
	then
	\begin{enumerate}
		\item\label{1 de CZ} For each $\lambda>0$ there exist a family of pairwise disjoint cubes $\{Q_j\}_{j\in{\mathbb{J}}}\subset{\mathcal{D}}$  such that 
		\begin{equation*}
			\lambda<\|\chi_{Q_j}\|_{\beta(\cdot)}\frac{\|f\chi_{Q_j}\|_{r(\cdot)}}{\|\chi_{Q_j}\|_{r(\cdot)}}\leq D^2_{r(\cdot)}\lambda,
		\end{equation*}
		and
		\begin{equation*}
			\Omega^{\mathcal{D}}_\lambda=\left\{x\in\R : M^{\mathcal{D}}_{{\beta(\cdot)},{r(\cdot)}}f(x)>\lambda\right\}=\underset{j\in{\mathbb{J}}}{\bigcup}Q_j.
		\end{equation*}
		
		\item\label{2 de CZ} If $\{Q_j\}_{j\in{\mathbb{J}}}$ are the cubes provided by $\eqref{1 de CZ}$ at the highest $\lambda$, then
		\begin{equation*}
			\left\{x\in\R : M_{{\beta(\cdot)},{r(\cdot)}}f(x)>2^nD^2_{\beta(\cdot)} D^3_{r(\cdot)}\lambda\right\}\subset\underset{j\in{\mathbb{J}}}{\bigcup}3Q_j.
		\end{equation*}
		
		\item \label{3 de CZ} There exists a constant $a_0$ such that if $a>a_0$ and $\left\{Q_j^k\right\}_{j\in{\mathbb{J}}}$ are the cubes provided by $\eqref{1 de CZ}$ at the highest $\lambda=a^k$, then the family $S=\left\{Q_j^k\right\}_{j\in{\mathbb{J}},k\in\mathbb{Z}}$ is sparse.
	\end{enumerate}
\end{lema}
The proof is in Section \ref{sec 5}.

In the following remark, we give a construction of exponents and some properties they satisfy that will be used in the proof of Theorem \ref{teo M}.
	
\begin{remark}\label{remark exp aux}
	Let ${p(\cdot)},{q(\cdot)},{r(\cdot)},{s(\cdot)},{\beta(\cdot)}$ be as in the hypotheses of Theorem \ref{teo M}. For $s_0\in\left( \max\left\{\frac{1}{p^-},\frac{1}{(q^\prime)^-},1-\frac{s^--1}{p^+}\right\},1\right)$, the functions $u(\cdot)=\frac{p(\cdot)}{p(\cdot)(1-s_0)+1}$ and $v^\prime(\cdot)=\frac{q^\prime(\cdot)}{q^\prime(\cdot)(1-s_0)+1}$ are well defined exponents, moreover $s_0>1-\frac{s^--1}{p^+}$ implies that $r(\cdot)<u(\cdot)$. Let $t(\cdot)$ be the exponent defined by $1/r(\cdot)=1/u(\cdot)+1/t(\cdot)$. By Lemma \ref{algebra de Plog}, ${u(\cdot)},{v(\cdot)},t(\cdot)\in\mathcal{P}^{\text{log}}(\R)$ and by a direct computation we have
	\begin{enumerate}
		\item $s_0{v(\cdot)}=(s_0{q^\prime(\cdot)})^\prime$.
		\item  $s_0{t(\cdot)}=\left[s_0\left(\frac{{p(\cdot)}{r(\cdot)}}{{p(\cdot)}-{r(\cdot)}}\right)^\prime\right]^\prime.$
		\item\label{remark3}
		$s_0>\frac{1}{2}$ and $\frac{1}{{v(\cdot)}}+\frac{1}{{t(\cdot)}}+2(1-s_0)=\frac{{\beta(\cdot)}-{r(\cdot)}}{{\beta(\cdot)}{r(\cdot)}}$.
		\end{enumerate}
	By Theorems \ref{acot MBr} and \ref{acot Mr}, the operator $M_{{\beta(\cdot)},{u(\cdot)}}$ is bounded from $L^{p(\cdot)}(\R)$ to $L^{q(\cdot)}(\R)$ and the operator $M_{v^\prime(\cdot)}$ is bounded on $L^{q^\prime(\cdot)}(\R)$.
\end{remark}
\begin{proof}[Proof of Theorem \ref{teo M}]
	The necessity follows immediately from Proposition~\ref{caracterizacion Apr} and from the fact that ${\bf A}_{{\beta(\cdot)},{r(\cdot)},Q}f(x)\leq M_{{\beta(\cdot)},{r(\cdot)}}f(x)$ for all $x\in\R$.
	 
	We now prove the sufficiency. Assume that $\omega\in \mathbb{A}_{{q(\cdot)},\frac{{\beta(\cdot)}{r(\cdot)}}{{\beta(\cdot)}-{r(\cdot)}}}$. By Lemma \ref{Apr implica Ap}, $\omega\in \mathcal{A}_{q(\cdot)}$ and $\omega^{-1}\in\mathcal{A}_{\frac{{p(\cdot)}{r(\cdot)}}{{p(\cdot)}-{r(\cdot)}}}$, or equivalently, $\omega^{-1}\in \mathcal{A}_{q^\prime(\cdot)}$ and $\omega\in\mathcal{A}_{\left(\frac{{p(\cdot)}{r(\cdot)}}{{p(\cdot)}-{r(\cdot)}}\right)^\prime}$. By \\Proposition \ref{Ap abierto izq} there exists $s_0\in\left( \max\left\{\frac{1}{p^-},\frac{1}{(q^\prime)^-},1-\frac{s^--1}{p^+}\right\},1\right)$ such that $\omega^{-1/s_0}\in \mathcal{A}_{s_0{q^\prime(\cdot)}}$ and $\omega^{1/s_0}\in\mathcal{A}_{s_0\left(\frac{{p(\cdot)}{r(\cdot)}}{{p(\cdot)}-{r(\cdot)}}\right)^\prime}$. Let ${u^\prime(\cdot)}$ and ${v(\cdot)}$ be the exponents defined as in Remark \ref{remark exp aux}. 
	
	Since $r^+<\infty$, by Lemma \ref{densidad en Lp de omega}, it is enough to consider functions belonging to $C_c^\infty$, which also satisfy \eqref{hipotesis de C-Z}. Indeed, let $f\in C^\infty_c$, we note that $\beta(\cdot)\geq r(\cdot)s(\cdot)$ with $s^->1$ implies that $\left(\frac{\beta r}{\beta-r}\right)_\infty<\infty$. Therefore, by Lemmas \ref{norma de cubos} and \ref{norma de cubos comparada con |Q|}, 
	\begin{equation*}
	\|\chi_Q\|_{\beta(\cdot)}\frac{\|f\chi_Q\|_{r(\cdot)}}{\|\chi_Q\|_{r(\cdot)}}\lesssim\|f\|_{r(\cdot)}\|\chi_Q\|^{-1}_{\frac{\beta(\cdot) r(\cdot)}{\beta(\cdot)-r(\cdot)}}\approx\|f\|_{r(\cdot)}|Q|^{-\frac{1}{\left(\frac{\beta r}{\beta-r}\right)_\infty}}\to 0
	\end{equation*}
	as $|Q|\to \infty$.
	
	By Lemma \ref{lema descomposición de C-Z}, for sufficiently large $a$, there exist dyadic cubes $\{Q_j^k\}$, pairwise disjoint measurable sets $E_j^k\subset Q_j^k$, and constants $\eta,C_0>0$ such that $|Q_j^k|\leq \eta |E_j^k|$,
	\begin{equation*}
		a^k<\left\|\chi_{Q_j^k}\right\|_{\beta(\cdot)}\frac{\left\|f\chi_{Q_j^k}\right\|_{r(\cdot)}}{\left\|\chi_{Q_j^k}\right\|_{r(\cdot)}}
	\end{equation*}
	and
	\begin{equation*}
		\Omega_k=\left\{x\in{\R}: M_{{\beta(\cdot)},{r(\cdot)}}f(x)>C_0a^k\right\}\subset\bigcup_{j,k}3Q_j^k.
	\end{equation*}
	To control $\|(Mf)\omega\|_{q(\cdot)}$, by Proposition \ref{propo norma conjugada} is enough to estimate $\int_{\R} \left(M_{{\beta(\cdot)},{r(\cdot)}}f\right)\omega g$ with $\|g\|_{q^\prime(\cdot)}\leq1$. If $\|g\|_{q^\prime(\cdot)}\leq 1$, then
	\begin{align*}
		\int_{\R} \left(M_{{\beta(\cdot)},{r(\cdot)}}f\right)\omega g&\leq \sum_{k}\int_{\Omega_k\setminus\Omega_{k+1}} \left(M_{{\beta(\cdot)},{r(\cdot)}}f\right)\omega g\leq C_0 a\sum_{k}a^k\int_{\Omega_k\setminus\Omega_{k+1}} \omega g\\
		&\lesssim \sum_{k,j}a^k\int_{3Q_j^k}\omega g\\
		&\leq \sum_{k,j} \left\|\chi_{Q_j^k}\right\|_{\beta(\cdot)}\frac{\left\|f\chi_{Q_j^k}\right\|_{r(\cdot)}}{\left\|\chi_{Q_j^k}\right\|_{r(\cdot)}} \frac{\left|3Q_j^k\right|}{\left|3Q_j^k\right|}\int_{3Q_j^k}\omega g\\
		&\lesssim \sum_{k,j} \left|E_j^k\right| \left\|\chi_{Q_j^k}\right\|_{\beta(\cdot)}\frac{\left\|f\chi_{Q_j^k}\right\|_{r(\cdot)}}{\left\|\chi_{Q_j^k}\right\|_{r(\cdot)}} \frac{1}{\left|3Q_j^k\right|}\int_{3Q_j^k}|\omega g|.
	\end{align*}
By Lemma \ref{norma de cubos}, $|3Q_j^k|\approx\|\chi_{3Q_j^k}\|_{v(\cdot)}\|\chi_{3Q_j^k}\|_{v^\prime(\cdot)}$ and $\|3Q_j^k\|_{r(\cdot)}\approx\|\chi_{3Q_j^k}\|_{u(\cdot)}\|\chi_{3Q_j^k}\|_{t(\cdot)}$. Applying Hölder's inequality \ref{teo Hölder} to $\left\|f\omega\omega^{-1}\chi_{Q_j^k}\right\|_{r(\cdot)}$ and $\int_{3Q_j^k}|\omega g|$,
	\begin{align*}
		\int_{\R} &\left(M_{{\beta(\cdot)},{r(\cdot)}}f\right)\omega g\\
		&\lesssim \sum_{k,j}\int_{E_j^k}\left\|\chi_{Q_j^k}\right\|_{\beta(\cdot)}\frac{\left\|f\omega\chi_{Q_j^k}\right\|_{u(\cdot)}}{\left\|\chi_{Q_j^k}\right\|_{u(\cdot)}}\frac{\|\omega^{-1}\chi_{Q_j^k}\|_{t(\cdot)}}{\|\chi_{Q_j^k}\|_{t(\cdot)}} \frac{\|\omega\chi_{3Q_j^k}\|_{v(\cdot)}}{\|\chi_{3Q_j^k}\|_{v(\cdot)}}\frac{\|g\chi_{3Q_j^k}\|_{v^\prime(\cdot)}}{\|\chi_{3Q_j^k}\|_{v^\prime(\cdot)}}\\
		&\leq \int_{{\R}}M_{{\beta(\cdot)},{u(\cdot)}}(f\omega)M_{v^\prime(\cdot)}(g) \frac{\|\omega\chi_{3Q_j^k}\|_{v(\cdot)}}{\|\chi_{3Q_j^k}\|_{v(\cdot)}} \frac{\|\omega^{-1}\chi_{Q_j^k}\|_{t(\cdot)}}{\|\chi_{Q_j^k}\|_{t(\cdot)}}.
	\end{align*}
	Assume for the moment that
	\begin{equation}\label{cond peso}
		\frac{\|\omega\chi_{3Q_j^k}\|_{v(\cdot)}}{\|\chi_{3Q_j^k}\|_{v(\cdot)}} \frac{\|\omega^{-1}\chi_{Q_j^k}\|_{t(\cdot)}}{\|\chi_{Q_j^k}\|_{t(\cdot)}}
	\end{equation}  
	is uniformly bounded for all cubes $Q$. Again, by Hölder's inequality \ref{teo Hölder} and Remark \ref{remark exp aux},
	\begin{align*}
		\int_{\R} \left(M_{{\beta(\cdot)},{r(\cdot)}}f\right)\omega g\lesssim \|M_{{\beta(\cdot)},{r(\cdot)}}(f\omega)\|_{q(\cdot)}\|M_{v^\prime(\cdot)}g\|_{q^\prime(\cdot)}\lesssim \|f\omega\|_{p(\cdot)}.
	\end{align*}
	
	Finally, to see the boundedness of \eqref{cond peso}, using Lemma \ref{doblante}, is enough to control it for any cube $Q$. By Remark \ref{remark exp aux}, $\omega^{-1/s_0}\in\mathcal{A}_{\left(s_0{v(\cdot)}\right)^\prime}=\mathcal{A}_{s_0{q^\prime(\cdot)}}$ and $\omega^{1/s_0}\in\mathcal{A}_{(s_0{t(\cdot)})^\prime}=\mathcal{A}_{s_0\left(\frac{{p(\cdot)}{r(\cdot)}}{{p(\cdot)}-{r(\cdot)}}\right)^\prime}$, therefore
	\begin{align*}
		\left\|\omega\chi_Q\right\|_{v(\cdot)}\left\|\omega ^{-1}\chi_Q\right\|_{t(\cdot)}&=\left(\left\|\omega^{1/s_0}\chi_Q\right\|_{s_0{v(\cdot)}}\left\|\omega^{-1/s_0}\chi_Q\right\|_{s_0{t(\cdot)}}\right)^{s_0}\\
		&\lesssim|Q|^{2s_0}\left(\left\|\omega^{-1}\chi_Q\right\|_{{q^\prime(\cdot)}}\left\|\omega\chi_Q\right\|_{\left(\frac{{p(\cdot)}{r(\cdot)}}{{p(\cdot)}-{r(\cdot)}}\right)^\prime}\right)^{-1}\\
		&\lesssim |Q|^{2s_0-2}\|\omega\chi_Q\|_{q(\cdot)}\|\omega^{-1}\chi_Q\|_{\frac{{p(\cdot)}{r(\cdot)}}{{p(\cdot)}-{r(\cdot)}}}\lesssim \frac{\|\chi_Q\|_{\frac{{\beta(\cdot)}{r(\cdot)}}{{\beta(\cdot)}-{r(\cdot)}}}}{|Q|^{2(1-s_0)}}\\
		&\lesssim \|\chi_Q\|_{v(\cdot)}\|\chi_Q\|_{t(\cdot)},
	\end{align*}
	where the last inequality follows from \eqref{remark3} of Remark \ref{remark exp aux}. 
\end{proof}

 \section{Hörmander-type Condition}\label{sec 4}
In this section, we present the proofs of Theorems \ref{teo Hör} and \ref{teo Hör frac}, which establish the domination of the operators $T$ and $T_{\beta(\cdot)}$ by the maximal operator $M_{\beta(\cdot), r(\cdot)}$, as well as weighted strong estimates for the integral operators. We also provide examples of kernels that satisfy the hypotheses of Theorems~\ref{teo Hör} and~\ref{teo Hör frac}.
\subsection{Proof of Theorems \ref{teo Hör} and \ref{teo Hör frac}}

\

Observe that if $p(\cdot),q(\cdot)\in\mathcal{P}^{\text{log}}(\R)$ and ${p(\cdot)}\leq {q(\cdot)}$, then by Hölder's inequality \ref{teo Hölder} and Lemma \ref{norma de cubos}, 
	\begin{equation}\label{comparacion de promedios}
		\frac{\|f\chi_Q\|_{p(\cdot)}}{\|\chi_Q\|_{p(\cdot)}}\lesssim\frac{\|f\chi_Q\|_{q(\cdot)}}{\|\chi_Q\|_{q(\cdot)}}.
	\end{equation}
Therefore
	\begin{equation*}
	M_{p(\cdot)}f(x)\lesssim M_{q(\cdot)}f(x)	
	\end{equation*}
	for all $x\in\R$, and the inclusions 
	\begin{equation*}
	H_{\infty,i}\subset H_{q(\cdot),i}\subset H_{p(\cdot),i}\subset\cdots \subset H_{1,i}
	\end{equation*}
	 holds.

\begin{remark}\label{T es debil (1,1)}
	Let $r(\cdot)\in \mathcal{P}^\text{log}(\R)$ and $K\in H_{r(\cdot)}\subset H_{1,2}$. If the operator $T$ given by the kernel is bounded on some $L^{p_0}$ with $1<p_0<\infty$, then it is of weak type $(1,1)$ (see Theorem 5.10 in \cite{Duo}).
\end{remark}

The following estimate for the Hardy-Littlewood maximal operator is well known:
\begin{teo}[\cite{Garcia-Cuerva}]\label{teo sharp}For any $0<p<\infty$ and $\omega\in A_\infty$ there exists a constant C such that
\begin{equation*}
	\int_\R (Mf(x))^p\omega(x)dx\leq C	\int_\R \left(M^\sharp f(x)\right)^p\omega(x)dx.
\end{equation*}	
whenever the left-hand side is finite.
\end{teo}

The key to the proof Theorem \ref{teo Hör} is a pointwise estimate of $(M^\sharp(Tf)^\delta)^{1/\delta}$ by $M_{r^\prime(\cdot)}f$, established in the following lemma.

\begin{lema} \label{acotación puntal del maximal sharp del operador} Let $r(\cdot)\in\mathcal{P}^{\log}(\R)$, and let $T$ be an integral operator bounded on some $L^{p_0}(\R)$ with $1<p_0<\infty$, whose kernel $K$ belongs to $H_{r(\cdot)}$. For all $0<\delta<1$, there exists a constant $C_\delta$ such that
	\begin{equation*}
	\left(M^\sharp\left|Tf\right|^\delta(x)\right)^\frac{1}{\delta}\leq C_\delta M_{r^\prime(\cdot)}f(x)
	\end{equation*}
	for all $x\in\R$.
\end{lema}
\begin{proof} Let $0<\delta<1$. Let $x_0$ be fixed and let $Q$ be any cube containing $x_0$. We must estimate
	\begin{equation}\label{a}
	\frac{1}{|Q|}\int_{Q}\left||Tf|^\delta - |a|^\delta \right|\leq\frac{1}{|Q|}\int_{Q}|Tf-a|^\delta
	\end{equation}
	to appropriate $a$.	Let ${\tilde{Q}}=2Q$, $f_1=f\chi_{\tilde{Q}}$, $f_2=f-f_1$, and $a=Tf_2(x_0)$. Therefore the right-side of \eqref{a} is dominated by
	\begin{equation*}
	\frac{1}{|Q|}\int_{Q}|Tf_1(x)|^\delta dx+\frac{1}{|Q|}\int_{Q}|Tf_2(x)-Tf_2(x_0)|^\delta dx=I+II.
	\end{equation*}
	
	For $I$,  by Remark \ref{T es debil (1,1)}, $T$ is weak type $(1,1)$, therefore by Kolmogorov’s inequality, 
	\begin{equation*}
	I\lesssim\left[\frac{1}{|{\tilde{Q}}|}\int_{{\tilde{Q}}}|f_1|\right]^\delta=\left[\frac{1}{|{\tilde{Q}}|}\int_{{\tilde{Q}}}|f|\right]^\delta\leq \left[Mf(x_0)\right]^\delta\lesssim\left[M_{{r^\prime(\cdot)}}f(x_0)\right]^\delta.
	\end{equation*}
	For $II$. For all $x\in Q$, by Hölder's inequality \ref{teo Hölder} and Lemma \ref{norma de cubos}
	\begin{align*}
	|Tf_2(&x)-Tf_2(x_0)|\\
	&\leq \int_{{\tilde{Q}}^c}\left|K(x,y)-K(x_0,y)\right||f(y)|dy\\
	&=\sum_{m=1}^{\infty}\frac{\left(2^m\ell({\tilde{Q}})\right)^n}{|2^m{\tilde{Q}}|}\int_{2^m{\tilde{Q}}}\left|K(x,y)-K(x_0,y)\right|\chi_{2^m{\tilde{Q}}\setminus2^{m-1}{\tilde{Q}}}|f(y)|dy\\	
	&\lesssim\sum_{m=1}^{\infty}\left(2^m\ell({\tilde{Q}})\right)^n\frac{\left\|K(x,\cdot)-K(x_0,\cdot)\chi_{2^m{\tilde{Q}}\setminus2^{m-1}{\tilde{Q}}}\right\|_{r(\cdot)}}{\|\chi_{2^m{\tilde{Q}}}\|_{r(\cdot)}}\frac{\|f\chi_{2^m{\tilde{Q}}}\|_{r^\prime(\cdot)}}{\|\chi_{2^m{\tilde{Q}}}\|_{r^\prime(\cdot)}}\\
	&\leq M_{r^\prime(\cdot)} f(x_0)\sum_{m=1}^{\infty}\left(2^m\ell({\tilde{Q}})\right)^n\frac{\left\|K(x,\cdot)-K(x_0,\cdot)\chi_{2^m{\tilde{Q}}\setminus2^{m-1}{\tilde{Q}}}\right\|_{r(\cdot)}}{\|\chi_{2^m{\tilde{Q}}}\|_{r(\cdot)}}
	\end{align*}
	Finally, since $K\in H_{r(\cdot)}\subset H_{{r(\cdot)},1}$.
	\begin{equation*}
	II\lesssim\frac{1}{|Q|}\int_{Q}\left[M_{r^\prime(\cdot)} f(x_0)\right]^\delta dx=\left[M_{r^\prime(\cdot)} f(x_0)\right]^\delta.
	\end{equation*}
\end{proof}
\begin{remark}\label{cambio de hipo}
	The conditions that $K \in H_{r(\cdot),2}$ and that $T$ is bounded on some $L^{p_0}(\mathbb{R})$ were only used to control $I = \frac{1}{|Q|}\int_Q |T_1 f(x)|^\delta dx$. These hypothesis can be replaced by the boundedness of $T$ on $L^{r^\prime(\cdot)}(\mathbb{R})$. In fact, by Jensen's inequality and \eqref{comparacion de promedios},
	\begin{equation*}
		I\leq \left(\frac{1}{|Q|}\int_Q |T_1 f(x)|  dx\right)^\delta\lesssim \left(\frac{\|Tf_1\chi_Q\|_{r^\prime(\cdot)}}{\|\chi_Q\|_{r^\prime(\cdot)}}\right)^\delta\lesssim  \left( \frac{\|f\chi_Q\|_{r^\prime(\cdot)}}{\|\chi_Q\|_{r^\prime(\cdot)}} \right)^\delta\leq \left(M_{r^\prime(\cdot)} f\right)^\delta.
	\end{equation*}
	
\end{remark}

To prove the part $\ref{b del teo Hör}$ of Theorem \ref{teo Hör} we need a extrapolation result. In \cite{Extrapol sin pesos} was established the first result on extrapolation in the context of unweighted variable Lebesgue spaces. Later, in \cite{Extrapol con pesos}, the following result on extrapolation was proved in the context of weighted Lebesgue spaces.

\begin{teo}[\cite{Extrapol con pesos}]\label{extrapol}
	Let $\mathcal{F}$ be a family of pairs of functions. Suppose that for some $p_0$, $1<p_0<\infty$, and every $\omega_0\in A_{p_0}$,
	\begin{equation*}
		\int_\R f(x)^{p_0}\omega_0(x)dx\leq C\int_\R g(x)^{p_0}\omega_0(x)dx 
	\end{equation*}
	for all $(f,g)\in \mathcal{F}$, whenever the left-hand side is finite. If $p(\cdot)\in\mathcal{P}^{\text{log}}(\R)$ with $1<p^-\leq p^+<\infty$ and $\omega\in\mathcal{A}_{p(\cdot)}$, then
	\begin{equation*}
		\|f\omega\|_{p(\cdot)}\leq C\|g\omega\|_{p(\cdot)}
	\end{equation*}
	for all $(f,g)\in \mathcal{F}$, whenever the left-hand side is finite.
\end{teo}
Now we are in a position to prove Theorem \ref{teo Hör}.
\begin{proof}[Proof of Theorem \ref{teo Hör}]
	For \ref{a del teo Hör}, suppose that $\int_{\R}|Tf|^p w<\infty$. Let $p$, with $0<p<\infty$, and $\omega\in A_\infty$, for a sufficiently small $\delta$, $0<\delta<1$, $\omega\in A_{p/\delta}$, so
	\begin{equation*}
		\int_{\R}\left(M|Tf|^\delta\right)^{\frac{p}{\delta}}\omega\lesssim\int_{\R}|Tf|^p \omega<\infty.
	\end{equation*}
	By Theorem \ref{teo sharp} and Lemma \ref{acotación puntal del maximal sharp del operador},
		\begin{equation*}
			\int_{\R}|Tf|^p \omega\leq\int_{\R}\left(M|Tf|^\delta\right)^\frac{p}{\delta}\omega
			\lesssim\int_{\R}\left(M^\sharp|Tf|^\delta\right)^\frac{p}{\delta}\omega
			\lesssim \int_{\R}\left(M_{r^\prime(\cdot)} f(x)\right)^p \omega.
		\end{equation*}
		
	On the other hand, \ref{b del teo Hör} is an immediate consequence of \ref{a del teo Hör} and Theorem \ref{extrapol} applied to the pairs of functions $(Tf, M_{r^\prime(\cdot)}f)$.

\end{proof}
\begin{proof}[Proof of Theorem \ref{teo T}]
	Since $\mathbb{A}_{p(\cdot),r^\prime(\cdot)}\subset\mathcal{A}_{p(\cdot)}$ by Lemma \ref{Apr implica Ap}. If\\ $\omega\in\mathbb{A}_{p(\cdot),r^\prime(\cdot)}$, by \ref{b del teo Hör} of Theorem \ref{teo Hör} and Corollary \ref{coro Mr}, then
	\begin{equation*}
		\|(Tf)\omega\|_{p(\cdot)}\lesssim		\|(M_{r^\prime(\cdot)}f)\omega\|_{p(\cdot)}\lesssim	\|f\omega\|_{p(\cdot)}
	\end{equation*}
	whenever the left-hand side is finite.
\end{proof}

We follow the same approach as in the proof of Theorem \ref{teo Hör} to prove Theorem \ref{teo Hör frac}.

\begin{lema}\label{sumabilidad de beta}
	Let $\beta(\cdot)\in\mathcal{P}^{\text{log}}(\R)$. If $\beta^+<\infty$ then 	\begin{equation*}
	\sum_{m=0}^{\infty}\|\chi_{2^{-m}Q}\|_{\beta(\cdot)}\lesssim \|\chi_Q\|_{\beta(\cdot)}.
	\end{equation*}
\end{lema}
\begin{proof}
	If $|Q|\leq1$, by Lemma \ref{norma de cubos comparada con |Q|},
	\begin{equation*}
	\sum_{m=0}^{\infty}\|\chi_{2^{-m}Q}\|_{\beta(\cdot)}\approx\sum_{m=0}^{\infty}|2^{-m}Q|^\frac{1}{\beta(c_Q)}\approx|Q|^{\frac{1}{\beta(c_Q)}}\approx \|\chi_Q\|_{\beta(\cdot)}.
	\end{equation*}
	If $|Q|>1$, let $N$ be such that $|2^{-N-1}Q|\leq 1<|2^{-N}Q|$. Hence $|2^{-m}Q|\leq 1$ if and only if $m> N$. By Lemma \ref{norma de cubos comparada con |Q|},
	\begin{align*}
	\sum_{m=0}^{\infty}\|\chi_{2^{-m}Q}\|_{\beta(\cdot)}&\approx\sum_{m=0}^{N}|2^{-m}Q|^\frac{1}{\beta_\infty}+\sum_{m=N+1}^{\infty}|2^{-m}Q|^\frac{1}{\beta(c_Q)}\\
	&\leq |Q|^{\frac{1}{\beta_\infty}}\sum_{m=0}^{N}2^{-\frac{mn}{\beta_\infty}}+\sum_{m=N+1}^{\infty}|2^{-m}Q|^\frac{1}{\beta^+}\\
	&\lesssim |Q|^\frac{1}{\beta_\infty}+|Q|^{\frac{1}{\beta^+}}\leq|Q|^{\frac{1}{\beta_\infty}}\approx\|\chi_Q\|_{\beta(\cdot)}.
	\end{align*}
\end{proof}

\begin{lema} \label{acotación puntal del maximal sharp del operador fraccionario} 
Let $r(\cdot),\beta(\cdot)\in\mathcal{P}^{\log}(\R)$ such that $\beta^+<\infty$, and let $T_{\beta(\cdot)}$ be an integral operator whose kernel $K$ belongs to $H_{\beta(\cdot),r(\cdot),1}\cap S_{\beta(\cdot),1,2}$. For all $0<\delta\leq1$, there exists a constant $C$ such that
	\begin{equation*}
	\left(M^\sharp\left|T_{\beta(\cdot)}f\right|^\delta(x)\right)^\frac{1}{\delta}\leq C M_{{\beta(\cdot)},{r^\prime(\cdot)}}f(x)
	\end{equation*}
	for all $x\in\R$.
\end{lema}
\begin{proof}
	Let $0<\delta\leq1$. Let $x_0$ be fixed and let $Q$ be any cube containing $x_0$. Let ${\tilde{Q}}=2Q$, $f_1=f\chi_{\tilde{Q}}$, $f_2=f-f_1$ and $a=T_{\beta(\cdot)}f_2(x_0)$.
	\begin{align*}
	\left(\frac{1}{|Q|}\int_Q\left||T_{\beta(\cdot)}f(x)|^\delta-|a|^\delta\right|dx\right)^\frac{1}{\delta}\\
	\leq\frac{1}{|Q|}\int_Q|T_{\beta(\cdot)}f_1(x)&|dx+\frac{1}{|Q|}\int_Q|T_{\beta(\cdot)}f_2(x)-T_{\beta(\cdot)}f_2(x_0)|dx\\
	=I+II.\ \qquad \qquad&
	\end{align*}
	For $I$: for each $y\in \R$, let $\tilde{Q}_y$ be the cube centered at $y$ with $\ell(\tilde{Q}_y)=\ell(\tilde{Q})$. Since $K\in S_{{\beta(\cdot)},1,2}$ and Lemma \ref{sumabilidad de beta},
	\begin{align*}
		I&=\frac{1}{|Q|}\int_Q\left|\int_{\tilde{Q}}K(x,y)f(y)dy\right|dx\lesssim\frac{1}{|\tilde{Q}|}\int_{\tilde{Q}}|f(y)|\int_{2\tilde{Q}_y}|K(x,y)|dx dy\\
		&=\frac{1}{|\tilde{Q}|}\int_{\tilde{Q}}|f(y)|\sum_{m=0}^{\infty}\int_{\R}|k(x,y)|\chi_{2^{-m+1}\tilde{Q}_y\setminus2^{-m}\tilde{Q}_y}(x)dxdy\\
		&\lesssim\frac{1}{|\tilde{Q}|}\int_{\tilde{Q}}|f(y)|\|\chi_{\tilde{Q}_y}\|_{\beta(\cdot)}dy\leq\frac{1}{|\tilde{Q}|}\int_{\tilde{Q}}|f(y)|\|\chi_{2\tilde{Q}}\|_{\beta(\cdot)}dy\\
		&\lesssim M_{{\beta(\cdot)},1}f(x_0)\lesssim  M_{{\beta(\cdot)},{r(\cdot)}}f(x_0).
	\end{align*}
	For $II$, we proceed analogously to the proof of Lemma \ref{acotación puntal del maximal sharp del operador}. Since $K\in H_{\beta(\cdot),r(\cdot),1}$, for all $x\in Q$,
	\begin{align*}
	|T&f_2(x)-Tf_2(x_0)|\\
	&\leq\sum_{m=1}^{\infty}\int_{2^m \tilde{Q}\setminus2^{m-1}\tilde{Q}}|K(x,y)-K(x_0,y)||f(y)|dy\\
	&\lesssim \sum_{m=1}^{\infty}\frac{(2^m\ell(\tilde{Q}))^n\|\chi_{2^m\tilde{Q}}\|_{\beta(\cdot)}\int_{2^m \tilde{Q}\setminus2^{m-1}\tilde{Q}}|K(x,y)-K(x_0,y)||f(y)|dy}{\|\chi_{2^m\tilde{Q}}\|_{r(\cdot)}\|\chi_{2^m\tilde{Q}}\|_{r^\prime(\cdot)}\|\chi_{2^m\tilde{Q}}\|_{\beta(\cdot)}}\\
	&\lesssim M_{{\beta(\cdot)},{r^\prime(\cdot)}}f(x_0)\sum_{m=1}^{\infty} \frac{(2^m\ell(\tilde{Q}))^n}{\left\|\chi_{2^m\tilde{Q}}\right\|_{\beta(\cdot)}}\frac{\left\|(K(x,\cdot)-K(x_0,\cdot))\chi_{2^m \tilde{Q}\setminus2^{m-1}\tilde{Q}}\right\|_{r(\cdot)}}{\left\|\chi_{2^m \tilde{Q}}\right\|_{r(\cdot)}}\\
	&\lesssim M_{{\beta(\cdot)},{r^\prime(\cdot)}}f(x_0).
	\end{align*}
 	Therefore
	\begin{equation*}
	II\lesssim \frac{1}{|Q|}\int_QM_{{\beta(\cdot)},{r(\cdot)}}f(x_0)dx=M_{{\beta(\cdot)},{r(\cdot)}}f(x_0).
	\end{equation*}
\end{proof}

\begin{proof}[Proof of Theorem \ref{teo Hör frac}] The proof is analogous to the proof of Theorem \ref{teo Hör}, using Lemma \ref{acotación puntal del maximal sharp del operador fraccionario} instead of Lemma \ref{acotación puntal del maximal sharp del operador}.
\end{proof}
\begin{proof}[Proof of Theorem \ref{teo TB}]
		Since $\mathbb{A}_{q(\cdot),\frac{\beta(\cdot) r^\prime(\cdot)}{\beta(\cdot)-r^\prime(\cdot)}}\subset\mathcal{A}_{q(\cdot)}$ by Lemma \ref{Apr implica Ap}. If $\omega\in\mathbb{A}_{q(\cdot),\frac{\beta(\cdot) r^\prime(\cdot)}{\beta(\cdot)-r^\prime(\cdot)}}$, by \ref{b del teo Hör} of Theorem \ref{teo Hör frac} and Theorem \ref{teo M}
	\begin{equation*}
	\|(T_{\beta(\cdot)}f)\omega\|_{q(\cdot)}\lesssim		\|(M_{\beta(\cdot),r^\prime(\cdot)}f)\omega\|_{q(\cdot)}\lesssim	\|f\omega\|_{p(\cdot)}
	\end{equation*}
	whenever the left-hand side is finite.
\end{proof}

\begin{proof}[Proof of Theorem \ref{borde}]
	Let $Q$ be a cube. We note that ${q(\cdot)}=\frac{\beta(\cdot)r^\prime(\cdot)}{{\beta(\cdot)}-{r^\prime(\cdot)}}\in\mathcal{P}^{\text{log}}(\R)$ and satisfies $1/r^\prime(\cdot)=1/\beta(\cdot)+1/q(\cdot)$. Therefore, for all $x\in Q$, by Hölder's inequality, Lemma \ref{norma de cubos}, and $\mathbb{A}_{\infty,q(\cdot)}$ condition,
	\begin{equation*}
		\omega(x) \|\chi_Q\|_{\beta(\cdot)}\frac{\|\chi_Q f\|_{r^\prime(\cdot)}}{\|\chi_Q\|_{r^\prime(\cdot)}}\lesssim  \frac{\|\chi_Q \omega\|_\infty}{\|\chi_Q\|_{q(\cdot)}}\|\chi_Q\omega^{-1}\|_{q(\cdot)}\|f\omega\|_{\beta(\cdot)}\lesssim \|f\omega\|_{\beta(\cdot)}.
	\end{equation*}
	Hence,
	\begin{equation*}
	\| (M_{{\beta(\cdot)},{r^\prime(\cdot)}}f)\omega\|_\infty\lesssim\|f\omega\|_{\beta(\cdot)}.
	\end{equation*}
	Finally, by Lemma \ref{acotación puntal del maximal sharp del operador fraccionario} for $\delta=1$,
	\begin{equation*}
	\|Tf\|_{*,\omega}\approx \|\omega M^\sharp Tf\|_\infty\lesssim \|\omega M_{{\beta(\cdot)},{r^\prime(\cdot)}}f\|_\infty\lesssim\|\omega f\|_{\beta(\cdot)}.
	\end{equation*}
	
\end{proof}

\subsection{Examples of kernels in $H_{r(\cdot)}$}

\

By adapting the example of the kernel presented in \cite{Hr es sharp}, we will see that the classes $H_{r(\cdot),i}$ are not empty and that the inclusions given in \eqref{contenciones de Hr} are strict. Although the examples are given for particular exponents, rescaling and translating, the result yield valid conclusions for general exponents $p(\cdot)\in\mathcal{P}^{\text{log}}(\mathbb{R})$. 

For $\beta>0$ let $K_1(t)=\chi_{[2,3]}(t)$ and $K_2(t)=t^{-\frac{1}{2}}\left[\log\left(\frac{e}{t}\right)\right]^{-\frac{1+\beta}{2}}\chi_{(0,1)}(t)$. We define the kernels $K$ and $\tilde{K}$ by
$$K(x,y)=K_1(x-y)K_2(y),$$
\begin{equation}\label{K tilde}
	\tilde{K}(x,y)=K_1(y)K_2(x-y-1).
\end{equation}
\begin{remark}\label{norma de K2} By a simple change of variables, $K_2(t)\in L^s(\mathbb{R})$ if and only if $s\leq 2$. Moreover, if ${r(\cdot)}\in \mathcal{P}(\mathbb{R})$ and $r(x)= r_0$ for all $x\in (0,\epsilon)$, for some $\epsilon>0$, then 
	$K_2\in L^{r(\cdot)}(\mathbb{R})$ if and only if $r_0\leq 2$.
\end{remark}
\begin{propo} [$H_{r(\cdot),1}\neq\emptyset$]\label{ejemplo de núcleo en Hr1}
	Let $r(\cdot)\in \mathcal{P}^{\log}(\mathbb{R})$. If $r(x)=2$ for all $x\in\left[0,1\right]$ and $r^-_{[0,16]}>1$, then $K\in H_{{r(\cdot)},1}$.
\end{propo}
\begin{proof}
	Let $Q=[a,b]$, let $x,y\in\frac{1}{2}Q$, and define $$A_m=\left\|\left[K(x,\cdot)-K(y,\cdot)\right]\chi_{2^mQ\setminus2^{m-1}Q}(\cdot)\right\|_{r(\cdot)}.$$
	
 	We claim that if $A_m\neq 0$, then $c_Q\in(0,16)$ and $|2^{m-1}Q|<32$. Assume that, let $m_0\in\mathbb{N}$ such that $|2^{m_0-1}Q|<32\leq|2^{m_0}Q|$. Since $A_m\leq2\|K_2\|_{r(\cdot)}=2\|K_2\|_2<\infty$,
	\begin{equation*}
	\sum_{m=1}^{\infty}2^m \ell(Q)\frac{\left\|\left[K(x,\cdot)-K(y,\cdot)\right]\chi_{2^mQ\setminus2^{m-1}Q}(\cdot)\right\|_{r(\cdot)}}{\|\chi_{2^m Q}\|_{r(\cdot)}}
	\lesssim \sum_{m=1}^{m_0}\frac{2^m \ell(Q)}{\|\chi_{2^m Q}\|_{r(\cdot)}}.
	\end{equation*} 
	
	We show that this sum is uniformly bounded for all $Q$. Note that $|2^{m_0}Q| = 2|2^{m_0-1}Q| < 64$. Let $N$ be such that $|2^NQ|<1\leq|2^{N+1}Q|$. By Lemmas \ref{norma de cubos} and \ref{norma de cubos comparada con |Q|},
	\begin{align*}\label{suma no depende del cubo Q}
	\sum_{m=1}^{m_0}\frac{2^m \ell(Q)}{\|\chi_{2^m Q}\|_{r(\cdot)}}&\approx\sum_{m=1}^{m_0}\|\chi_{2^m Q}\|_{r^\prime(\cdot)}\\
	&\approx \sum_{m=1}^{N}|2^mQ|^{\frac{1}{r^\prime\left(c_Q\right)}} + \sum_{m=N+1}^{m_0}|2^mQ|^{\frac{1}{r^\prime_\infty}}\\\
	&=|2^{m_0}Q|^{\frac{1}{r^\prime\left(c_Q\right)}}\sum_{m=1}^{m_0}(2^{m-m_0})^{\frac{1}{r^\prime\left(c_Q\right)}} + |2^{m_0}Q|^{\frac{1}{r^\prime_\infty}}\sum_{m=1}^{m_0}(2^{m-m_0})^{\frac{1}{r_\infty^\prime}}\\
	&\leq64\left[\sum_{n=0}^{\infty}\left(\frac{1}{2}\right)^{\frac{n}{r^\prime(c_Q)}} + \sum_{n=0}^{\infty}\left(\frac{1}{2}\right)^\frac{n}{r^\prime_\infty} \right]\\
	&\leq64\left[\sum_{n=0}^{\infty}\left(\frac{1}{2}\right)^{\frac{n}{\left(r^-_{(0,16)}\right)^\prime}} + \sum_{n=0}^{\infty}\left(\frac{1}{2}\right)^\frac{n}{r^\prime_\infty} \right]<\infty.
	\end{align*}

To prove the claim note that, by the definition of $A_m$, if $(0,1)\subset Q$, then $A_m=0$. Also, if $\frac{1}{2}Q\cap(2,4)=\emptyset$, then $A_m=0$. In fact, let $x_0\in\frac{1}{2}Q=\left[\frac{3a+b}{4},\frac{a+3b}{4}\right]$. If $\frac{1}{2}Q\cap(2,4)=\emptyset$, then $x_0\geq4$ or $x_0\leq2$. Therefore $\chi_{[2,3]}(x_0-z)=0$ for all $z\in(0,1)$ and consequently, $A_m=0$. 

Therefore, if $a\leq 0$ and $\frac{1}{2}Q\cap(2,4)\neq\emptyset$, then $(0,1)\subset Q$. If $4\leq\frac{3a+b}{4}$ and $a>0$, then $\frac{1}{2}Q\cap(2,4)=\emptyset$. So, we only need to consider $a>0$ and $b<16$ which implies $c_Q\in(0,16)$. Finally, since $c_{2^mQ}=c_Q\in (0,16)$, if $|2^{m-1}Q|\geq32$, then $(0,1)\subset 2^{m-1}Q$, and $(2^mQ\setminus2^{m-1}Q)\cap (0,1)=\emptyset$ and consequently, $A_m=0$.
\end{proof}

\begin{propo}($H_{{s(\cdot)},1}\varsubsetneq H_{{r(\cdot)},1}$) Let $s(\cdot)\in\mathcal{P} (\mathbb{R})$. If $s(x)=s_0>2$ for all $x\in[0,1]$, then $K\notin H_{s(\cdot),1}$.
	
\end{propo}
\begin{proof}
	Let $Q=[1,3]$, note that $(0,1)\subset2Q\setminus Q$, and for all $x\in\left(2,\frac{5}{2}\right)$ and $y\in\left(\frac{3}{2},2\right)$  we have,  $x,y\in\frac{1}{2}Q=\left[\frac{3}{2},\frac{5}{2}\right]$. Hence, $y-z<2$ for all $z\in (0,1)$. Therefore,
	\begin{equation*}
	\left\|\left[K(x,\cdot)-K(y,\cdot)\right]\chi_{2Q\setminus Q}(\cdot)\right\|_{s(\cdot)}
	=\left\|\chi_{[2,3]}(x-\cdot)K_2(\cdot)\right\|_{s_0}\\
	\geq\|\chi_{[0,x-2]}K_2 \|_{s_0}.
	\end{equation*}
	By Remark \ref{norma de K2}, the last term is not finite.
\end{proof}

For $r(\cdot)$ as in Proposition \ref{ejemplo de núcleo en Hr1}, although we do not prove that $K \in H_{r(\cdot),2}$, and thus that Lemma \ref{acotación puntal del maximal sharp del operador} (and consequently Theorem \ref{teo Hör}) holds, we will show that $K$ satisfies Remark \ref{cambio de hipo}. Moreover, the operator $T$ induced by $K$ satisfies that the left-hand side of the inequality in Theorem \ref{teo T} is finite for $f \in C_c^\infty$, as will see in the following propositions.

\begin{propo}\label{fuerte rprima}
	Let $r(\cdot)\in \mathcal{P}^{\log}(\mathbb{R})$. If $r(x)=2$ for all $x\in\left[0,1\right]$, then the operator $T$ induced by $K$ is bounded on $L^{r^\prime(\cdot)} (\mathbb{R})$.

\end{propo} 

\begin{proof}
	Let $f\in L^{r^\prime(\cdot)}(\mathbb{R})$, by Hölder's inequality,
	\begin{equation*}
	\|Tf\|_{r^\prime(\cdot)}\leq\left\| \left\|K_1(x-y)K_2(y) \right\|_{r(y)}\right\|_{r^\prime(x)}\left\|f\right\|_{r^\prime(\cdot)}.
	\end{equation*}
	
	Also, for $y\in supp[K_1(x-\cdot)K_2(\cdot)]$, we have $K_1(x-y)=\chi_{[2,3]}(x-y)\leq\chi_{[2,4]}(x)$. Therefore,
	\begin{equation*}
	\left\| \left\|K_1(x-y)K_2(y) \right\|_{r(y)}    \right\|_{r^\prime(x)}\leq\|\chi_{[2,4]}\|_{r^\prime(\cdot)}\|K_2\|_{r(\cdot)}\lesssim\|K_2\|_2<\infty.
	\end{equation*}
\end{proof}

\begin{propo}\label{lado izq finito K}
	Let $p(\cdot)\in \mathcal{P}(\mathbb{R})$, $T$ be the operator induced by $K$, and $\omega\in \mathcal{A}_{p(\cdot)}$. If $f\in C^\infty_c$, then $\|(Tf)\omega\|_{p(\cdot)}<\infty$.
\end{propo}
\begin{proof}
	Let $f\in C^\infty_c$. Note that if $y\in(0,1)$, then $x-y\in [2,3]$ implies $x\in[2,4]$. Therefore, by Hölder's inequality \ref{teo Hölder},
	\begin{equation*}
		\|(Tf)\omega\|_{p(\cdot)}\leq \left\|\omega(\cdot)\int_{\mathbb{R}}|\chi_{[2,3]}(\cdot-y)K_2(y)f(y)|dy\right\|_{p(\cdot)}\lesssim\|\omega\chi_{[2,4]}\|_{p(\cdot)}\|K_2\|_2\|f\|_2.
	\end{equation*}
	By Remark \ref{norma de K2}, the last term is finite.
\end{proof}

The following theorem is deduced from Propositions \ref{fuerte rprima} and \ref{lado izq finito K}, Remark \ref{cambio de hipo}, and Theorem \ref{teo T}.
\begin{teo}
	Let ${p(\cdot)},{r(\cdot)},s(\cdot)\in\mathcal{P}^{\log}(\mathbb{R})$ such that, $p^+<\infty$, $r(x)=2$ for all $x\in[0,1]$, $r^-_{[0,16]}>1$, and ${p(\cdot)}={r^\prime(\cdot)}{s(\cdot)}$ with $s^->1$. Let $T$ the integral operator given by the kernel $K$. If $\omega\in \mathbb{A}_{{p(\cdot)},{r^\prime(\cdot)}}$ then  
	$$\left\|(T f)\omega\right\|_{p(\cdot)}\lesssim \|f\omega\|_{p(\cdot)}.$$
	
\end{teo}
We now show that the kernel $\tilde{K}$ belongs to $H_{r(\cdot)}$ for an exponent $r(\cdot)\in\mathcal{P}^{\text{log}}(\mathbb{R})$.

\begin{propo} ($H_{r(\cdot)}\neq\emptyset$)\label{ejemplo de núcleo en Hr} Let $r(\cdot)\in\mathcal{P}^{\text{log}}(\mathbb{R})$. If $r(x)=2$ for all $x\in\left[2,5\right]$ and $r^-_{[-7,14]}>1$, then $\tilde{K}\in H_{r(\cdot)}$.
	
\end{propo}

\begin{proof}
	 Let $Q=[a,b]$, and let $x,y\in\frac{1}{2}Q$. Since 
	 $$A_m=\left\|\left[\tilde{K}(x,\cdot)-\tilde{K}(y,\cdot)\right]\chi_{2^mQ\setminus2^{m-1}Q}(\cdot)\right\|_{r(\cdot)}\leq 2\|K_2\|_2$$
	 and
	 $$B_m=\left\|\left[\tilde{K}(\cdot, x)-\tilde{K}(\cdot,y)\right]\chi_{2^mQ\setminus2^{m-1}Q}(\cdot)\right\|_{r(\cdot)}\leq 2\|K_2\|_2.$$
	 Following the same approach as in the proof of Proposition \ref{ejemplo de núcleo en Hr1}, it is enough to prove that	 
	\begin{enumerate}
		\item\label{Am=0} If $A_m\neq0$, then $c_Q\in(2,14)$ and $|2^{m-1} Q|<24$.
		\item\label{Bm=0} If $B_m\neq0$, then $c_Q\in(-7,5)$ and $|2^{m-1}Q|<24$.
	\end{enumerate}
	For \eqref{Am=0}, suppose that $A_m\neq0$. Then, for some $z\in[2,3]$, $x-z-1 \in (0,1)$ or $y-z-1\in(0,1)$. Without loss of generality, assume that $3<x<5$. Also, $[2,3] \subset Q$ implies $A_m=0$, so we only need to consider $a>2$ or $b<3$. Since $x\in\frac{1}{2}Q=\left[\frac{3a+b}{4},\frac{3b+a}{4}\right]$, if $a>2$, then $b<14$. If $b<3$ then $a>3$ which is absurd. Therefore, $c_Q=c_{2^m Q}\in(2,14)$, which implies that if $|2^{m-1} Q|\geq24$, then $[2,3]\subset2^{m-1}Q$ and consequently, $A_m=0$.
	
	For \eqref{Bm=0}, note that if $(3,5)\subset Q$ then $B_m=0$. In fact, suppose that $(3,5)\subset Q$ and $B_m\neq0$. Then, without loss of generality, assume that $K_1(x)K_2(z-x-1)\chi_{2^mQ\setminus2^{m-1}Q}(z)$ does not vanish for some $z$. Then there exists $z$ such that $x\in [2,3]$, $z-x-1\in(0,1)$, and $z\notin(3,5)$, which is a contradiction. Therefore, $a>3$ or $b<5$. Since $x\in [2,3]$ or $y\in[2,3]$, and $x,y\in\frac{1}{2}Q=\left[\frac{3a+b}{4},\frac{3b+a}{4}\right]$, if $a>3$, then $b\leq3$, which is a contradiction. If $b<5$, then $-7\leq a$. So $c_Q=c_{2^mQ}\in(-7,5)$, therefore if $|2^{m-1}Q|\geq24$, then $(3,5)\subset2^{m-1}Q$ and consequently, $B_m=0$.
\end{proof}	 

\begin{propo}($H_{s(\cdot)}\varsubsetneq H_{r(\cdot)}$)
	Let ${s(\cdot)}\in\mathcal{P}(\mathbb{R})$. If $s(x)=s_0>2$ for all $x\in [4,5]$, then $\tilde{K}\notin H_{{s(\cdot)},2}$. In particular, $\tilde{K}\notin H_{s(\cdot)}$.
\end{propo}
\begin{proof}
	To show this, it is enough to see that $\left\| \left[\tilde{K}(\cdot,x) - \tilde{K}(\cdot,y)\right]\chi_{2Q\setminus Q}(\cdot)\right\|_{r(\cdot)}$ is infinite for some cube $Q$ and $x,y\in\frac{1}{2}Q$. In fact, let $Q=[0,4]$, $x\in\left[1,2\right)$, and $y=3$. Since $(4,5)\subset2Q\setminus Q$, 
	\begin{equation*}
		\left\| \left[\tilde{K}(\cdot,x) - \tilde{K}(\cdot,3)\right]\chi_{2Q\setminus Q}(\cdot)\right\|_{{s(\cdot)}}\geq\left\| K_2(\cdot-4)\chi_{(4,5)}(\cdot)\right\|_{s(\cdot)}=
		\left\| K_2\right\|_{s_0}.
	\end{equation*}
	The last term is infinite by Lemma \ref{norma de K2}.
\end{proof} 

The operator $T$ induced by the kernel $\tilde{K}$ satisfies the hypotheses of Theorem \ref{teo T}. Moreover, the left-hand side of the inequality in this theorem is finite for $f\in C_c^\infty$.

\begin{propo}\label{lado izq finito K tilde}
	Let $p(\cdot)\in \mathcal{P}(\mathbb{R})$, $T$ be the operator induced by $\tilde{K}$, and $\omega\in \mathcal{A}_{p(\cdot)}$. Then
	\begin{enumerate}
		\item $T$ is bounded on $L^2(\mathbb{R})$.
		\item If $f\in C^\infty_c$, then $\|(Tf)\omega\|_{p(\cdot)}<\infty$.
	\end{enumerate}

\end{propo}
\begin{proof}
	The proof is analogous to that of Propositions \ref{fuerte rprima} and \ref{lado izq finito K}.
\end{proof}
The following theorem follows from Proposition \ref{lado izq finito K tilde} and Theorem \ref{teo T}.
\begin{teo}
	Let ${p(\cdot)},{r(\cdot)},s(\cdot)\in\mathcal{P}^{\log}(\mathbb{R})$ such that, $p^+<\infty$, $r(x)=2$ for all $x\in[2,5]$, $r^-_{[-7,14]}>1$, and ${p(\cdot)}={r^\prime(\cdot)}{s(\cdot)}$ with $s^->1$. Let $T$ the integral operator given by the kernel $\tilde{K}$. If $\omega\in \mathbb{A}_{{p(\cdot)},{r^\prime(\cdot)}}$, then  
	$$\left\|(T f)\omega\right\|_{p(\cdot)}\lesssim \|f\omega\|_{p(\cdot)}.$$
	
\end{teo}

\subsection{Construction of kernels in $H_{n/\alpha,r(\cdot),1}\cap \mathcal{S}_{n/\alpha,r(\cdot),2}$}

\

Following the approach of \cite{S_A}, we will build kernels that satisfy the hypotheses of Theorem~\ref{teo Hör frac} for ${\beta(\cdot)}\equiv n/\alpha$ with $0<\alpha<n$.
\begin{remark}

	Note that for ${r(\cdot)},{s(\cdot)}\in\mathcal{P}^{\text{log}}(\R)$ and for ${\beta(\cdot)}\in \mathcal{P}(\R)$ such that  ${r(\cdot)}\leq{s(\cdot)}$, by \eqref{comparacion de promedios}, we set $i=1,2$:
	\begin{equation}\label{contencion de S}
			S_{\beta(\cdot),s(\cdot),i}\subset S_{\beta(\cdot),r(\cdot),i}\subset S_{\beta(\cdot),1,i}.
	\end{equation}
\end{remark}	

\begin{propo}\label{nucleo en S_Br}
	Let $r(\cdot)\in\mathcal{P}^{\text{log}}(\R)$ be such that $0\leq\alpha<n$, and let $\bar{K}(x,y)=|x-y|^\alpha K(x,y)$. If $K\in S_{\infty,r(\cdot),i}$, $i=1,2$, then $\bar{K}\in S_{\frac{n}{\alpha},{r(\cdot)},i}$, with the convention $\frac{n}{0}=\infty$. In particular, $\bar{K}\in S_{\frac{n}{\alpha},1,i}$.
\end{propo}
\begin{proof}
	For $i=1$. If $\alpha=0$ there is noting to prove. If $\alpha>0$, let $Q$ be a cube. For all $x\in\frac{1}{2}Q$ and for all $y\in 2Q$, $|x-y|^\alpha\leq (\sqrt{n}\ell(Q))^\alpha$. Therefore 
	\begin{equation*}
	\frac{\|\bar{K}(x,\cdot)\chi_{2Q\setminus Q}(\cdot)\|_{r(\cdot)}}{\|\chi_{Q}\|_{r(\cdot)}}\lesssim \left(\sqrt{n}\right)^\alpha \ell(Q)^\alpha |Q|^{-1}\lesssim\frac{\|\chi_{Q}\|_{n/\alpha}}{|Q|}. 
	\end{equation*}	
	The case $i=2$ is analogous.
\end{proof}
\begin{propo}\label{construcion kernel}
	Let $r(\cdot)\in\mathcal{P}^{\text{log}}(\R)$ be such that $0<\alpha<n$, and let $\bar{K}(x,y)=|x-y|^\alpha K(x,y)$. If $K\in H_{{r(\cdot)},1}\cap S_{\infty,{r(\cdot)}}$, then $\bar{K}\in H_{\frac{n}{\alpha},{r(\cdot)},1}\cap S_{\frac{n}{\alpha},{r(\cdot)}}$. In particular, $\bar{K}\in H_{\frac{n}{\alpha},{r(\cdot)},1}\cap S_{\frac{n}{\alpha},1,2}$.
\end{propo}
\begin{proof}
	Let $K\in H_{{r(\cdot)},1}\cap S_{\infty,{r(\cdot)}}$, by Proposition \ref{nucleo en S_Br} it is enough to see that $\bar{K}\in H_{\frac{n}{\alpha},{r(\cdot)},1}$. Let $Q$ be a cube, $x,y\in\frac{1}{2}Q$, and $z\in2^mQ\setminus2^{m-1}Q$. By the mean value theorem, $|y-z|^\alpha-|x-z|^\alpha\lesssim \frac{(2^m \ell(Q))^\alpha}{2^m}$. Therefore,
	\begin{align*}
	|\bar{K}(x,z)&-\bar{K}(y,z)|\\
	&\lesssim ||x-z|^\alpha-|y-z|^\alpha||K(x,z)|+|y-z|^\alpha |K(x,z)-K(y,z)|\\
	&\lesssim (2^m \ell(Q))^\alpha\left( \frac{|K(x,z)|}{2^m}+|K(x,z)-K(y,z)|  \right).
	\end{align*}
	Finally, since $K\in H_{r(\cdot),1}\cap S_{\infty,r(\cdot),1}$,
	\begin{align*}
	\sum_{m=1}^{\infty}\frac{\left(2^m \ell(Q)\right)^n}{\|\chi_{2^mQ}\|_{\frac{n}{\alpha}}}&\frac{\left\|(\bar{K}(x,\cdot)-\bar{K}(y,\cdot))\chi_{2^mQ\setminus2^{m-1}Q}\right\|_{r(\cdot)}}{\|\chi_{2^mQ}\|_{r(\cdot)}}\\
	&\lesssim \sum_{m=1}^{\infty}\frac{(2^m\ell(Q))^{n}}{2^m}\frac{\|K(x,\cdot)\chi_{2^mQ\setminus2^{m-1}Q}\|_{r(\cdot)}}{\|\chi_{2^mQ}\|_{r(\cdot)}}\\
	&\quad+\sum_{m=1}^{\infty}\left(2^m\ell(Q)\right)^n\frac{\|(K(x,\cdot)-K(y,\cdot))\chi_{2^mQ\setminus2^{m-1}Q}\|_{r(\cdot)}}{\|\chi_{2^mQ}\|_{r(\cdot)}}\\
	&\lesssim \sum_{m=1}^{\infty}\frac{(2^m\ell(Q))^{n}}{2^m}|2^mQ|^{-1}+C<\infty.
	\end{align*}	
\end{proof}

We will see next that the kernel $\tilde{K}$ defined in \eqref{K tilde} lies in $K_{r(\cdot),1}\cap S_{\infty,r(\cdot)}$. Moreover, the operator $T_{n/\alpha}$ induced by $\bar{K}(x,y)=|x-y|^\alpha\tilde{K}(x,y)$ satisfies that the left-hand side of inequality of Theorem \ref{teo Hör frac} is finite for $f\in C_c^\infty$.

\begin{propo}\label{k tilde en s infinito}
	Let $\tilde{K}$ be the kernel defined in \eqref{K tilde} and let $r(\cdot)\in\mathcal{P}(\mathbb{R})$. If $r(x)=2$ for all $x\in[-7,14]$ then $\tilde{K}\in S_{\infty,r(\cdot)}$.
\end{propo}
\begin{proof}
	In the proof of Proposition \ref{ejemplo de núcleo en Hr}, it was implicitly proved that for all $x\in \frac{1}{2}Q$, $\|\tilde{K}(x,\cdot)\chi_{2Q\setminus Q}\|_{r(\cdot)}\neq0$ implies $Q\subset [-7,14]$. Therefore we need only consider this case. Let $Q\subset[-7,14]$, hence $|Q|^{-1/2}\lesssim|Q|^{-1}$. For all $x\in\frac{1}{2}Q$, 
	\begin{equation*}
	\frac{\|\tilde{K}(x,\cdot)\chi_{2Q\setminus Q}\|_{r(\cdot)}}{\|\chi_Q\|_{r(\cdot)}}\leq\frac{\|K_2\|_2}{|Q|^{1/2}}\lesssim |Q|^{-1}.
	\end{equation*}
	The proof that $\tilde{K}\in S_{\infty,r(\cdot),2}$ is analogous.
\end{proof}

\begin{propo}\label{lado izq finito K alpha}
	Let $p(\cdot)\in \mathcal{P}(\mathbb{R})$, $T_{n/\alpha}$ be the operator induced by $|x-y|^\alpha\tilde{K}(x,y)$, and $\omega\in \mathcal{A}_{p(\cdot)}$. If $f\in C^\infty_c$, then $\|(T_{n/\alpha}f)\omega\|_{p(\cdot)}<\infty$.
\end{propo}

\begin{proof}
	By the same argument of Proposition \ref{lado izq finito K}, we have
	\begin{equation*}
	\|(T_{n/\alpha}f)\omega\|_{p(\cdot)}\lesssim3^\alpha\|\omega\chi_{[2,4]}\|_{p(\cdot)}\|K_2\|_2\|f\|_2<\infty.
	\end{equation*}
\end{proof}
 The following theorem follows from Propositions \ref{ejemplo de núcleo en Hr}, \ref{construcion kernel}, \ref{k tilde en s infinito}, and \ref{lado izq finito K alpha}, and \\Theorem \ref{teo TB}.
\begin{teo}
	Let $p(\cdot),q(\cdot),r(\cdot),s(\cdot)\in\mathcal{P}^{\log}(\mathbb{R})$ such that $p(\cdot)<q(\cdot)\leq q^+<\infty$, $p(\cdot)=r^\prime(\cdot)s(\cdot)$ with $s^->1$, $r(x)=2$ for all $x\in[-7,14]$, and $\alpha/n=1/p(\cdot)-1/q(\cdot)$ with $0<\alpha<n$. Let $T_{n/\alpha}$ be the integral given by the kernel $\bar{K}(x,y)=|x-y|^\alpha\tilde{K}(x,y)$. If $\omega\in \mathbb{A}_{q(\cdot),\frac{\frac{n}{\alpha} r^\prime(\cdot)}{\frac{n}{\alpha}-r^\prime(\cdot)}}$, then 
	$$\left\|\left(T_{n/\alpha}f\right)\omega\right\|_{q(\cdot)}\lesssim\|f\omega\|_{p(\cdot)}.$$
\end{teo}

\section{Appendix}\label{sec 5} In this section, we present the proof of the equivalence between the norms $\|\cdot\|_{p(\cdot)}^*$ and $\|\cdot\|_{p(\cdot)}$, as well as Calderón-Zygmund decomposition. To prove Proposition \ref{propo norma conjugada}, we need two preliminary results.
\begin{lema}\label{lema acotación de la modular para q finito}
	Let ${p(\cdot)},{r(\cdot)}\in\mathcal{P}(\R)$ with $r(\cdot)\leq p(\cdot)<\infty$, and let ${q(\cdot)}$ be the exponent defined by $1/r(\cdot)=1/p(\cdot)+1/q(\cdot)$.	If $\left\|f\chi_{\R\setminus(\R)_\infty^{q(\cdot)}}\right\|^\ast_{p(\cdot)}\leq1$ and $\rho_{p(\cdot)}\left(f\chi_{\R\setminus(\R)_\infty^{q(\cdot)}}\right)<\infty$, then $\rho_{p(\cdot)}\left(f\chi_{\R\setminus(\R)_\infty^{q(\cdot)}}\right)\leq1.$	
\end{lema}
\begin{proof}
	Suppose that $\rho_{p(\cdot)}\left(f\chi_{\R\setminus(\R)_\infty^{q(\cdot)}}\right)>1$, then by Proposition \ref{propo modular} there exists $\lambda>1$ such that
	$$\rho_{p(\cdot)}\left(\frac{f\chi_{\R\setminus(\R)_\infty^{q(\cdot)}}}{\lambda}\right)=1.$$
	Define
	$$	 g(x)= \left\{ \begin{array}{lcc} \left(\frac{|f(x)|}{\lambda}\right)^\frac{p(x)-r(x)}{r(x)}\chi_{\R\setminus(\R)^{q(\cdot)}_\infty} & \text{if} & f(x)\neq0,  \\ \\ 0 & \text{if} & f(x)=0. \end{array} \right.$$  
	Then 
	\begin{equation*}
		\rho_{q(\cdot)}(g)=\int_{\chi_{\R\setminus(\R)^{q(\cdot)}_\infty}}\left(\frac{|f(x)|}{\lambda}\right)^{p(x)}dx\leq\rho_{p(\cdot)}\left(\frac{f\chi_{\R\setminus(\R)^{q(\cdot)}_\infty}}{\lambda}\right)=1,
	\end{equation*}
	which implies that $\|g\|_{p(\cdot)}\leq 1$. Therefore, by the hypothesis, definition of $\|\cdot\|^*_{p(\cdot)}$ and Proposition \ref{propo modular}
	\begin{align*}
		1&\geq\left\|f\chi_{\R\setminus(\R)_\infty^{q(\cdot)}}\right\|^\ast_{p(\cdot)}\geq\left\|f\chi_{\R\setminus(\R)_\infty^{q(\cdot)}}g\right\|_{r(\cdot)}
		\geq\rho_{r(\cdot)}\left(f\chi_{\R\setminus(\R)_\infty^{q(\cdot)} }g\right)\\
		&\geq\int_{\R\setminus(\R)_\infty^{q(\cdot)}}|f(x)|^{r(x)}\left(\frac{|f(x)|}{\lambda}\right)^{\frac{p(x)-r(x)}{r(x)}r(x)}dx\geq\lambda\rho_{p(\cdot)}\left(\frac{f\chi_{\R\setminus(\R)_\infty^{q(\cdot)}}}{\lambda}\right)=\lambda,
	\end{align*}
	which is a contradiction.
\end{proof}
\begin{lema}\label{lema acotación de la modular para p=r}
	Let ${p(\cdot)},{r(\cdot)}\in\mathcal{P}(\R)$ such that $r(\cdot)\leq p(\cdot)<\infty$ with $r^+<\infty$. Let ${q(\cdot)}$ by the exponent defined by $1/r(\cdot)=1/p(\cdot)+1/q(\cdot)$.	If $\left\|f\right\|^\ast_{p(\cdot)}\leq1$ and $\rho_{p(\cdot)}\left(f\chi_{\R\setminus(\R)_\infty^{q(\cdot)}}\right)<\infty$, then $\rho_{p(\cdot)}\left(k_{q(\cdot)}^{r^+}f\right)\leq1$, where
	$$k_{q(\cdot)}^{-1}=\left\|\chi_{\R\setminus(\R)_\infty^{q(\cdot)}}\right\|_\infty+\left\|\chi_{(\R)_\infty^{q(\cdot)}}\right\|_\infty.$$
\end{lema}
\begin{proof}
	By Lemma \ref{lema acotación de la modular para q finito}, $\rho_{p(\cdot)}\left(f\chi_{\R\setminus(\R)_\infty^{q(\cdot)}}\right)\leq1$. Therefore
	\begin{equation}\label{A1}
		\rho_{p(\cdot)}\left(f\chi_{\R\setminus(\R)_\infty^{q(\cdot)}}\right)\leq\left\| \chi_{\R\setminus(\R)^{q(\cdot)}_\infty}\right\|_\infty.
	\end{equation} 
	Define
	$$	 g(x)= \left\{ \begin{array}{lcc} k_{q(\cdot)}|f(x)|^\frac{p(x)-r(x)}{r(x)} & si & q(x)<\infty,\\ \\
	k_{q(\cdot)} & si & q(x)=\infty.\end{array} \right. $$
	By $(\ref{A1})$,
	\begin{align*}
		\rho_{q(\cdot)}\left(\frac{g}{k_{q(\cdot)}}\right)\leq\left\|\chi_{\R\setminus(\R)_\infty^{q(\cdot)}}\right\|_\infty+\left\|\chi_{(\R)^{q(\cdot)}_\infty}\right\|_\infty=k_{q(\cdot)}^{-1}.
	\end{align*}
	Since $k_{q(\cdot)}\leq 1$, Proposition \ref{propo modular} implies that $\|g\|_{q(\cdot)}\leq1$.
	
	Finally, again by Proposition \ref{propo modular},
	\begin{align*}
		1&\geq\|f\|^\ast_{p(\cdot)}\geq\|fg\|_{r(\cdot)}\geq\rho_{r(\cdot)}(fg)\\
		&= \int_{\R\setminus(\R)^{q(\cdot)}_\infty}k_{r(\cdot)}^{r(x)}|f(x)|^{r(x)}|f(x)|^{\frac{p(x)-r(x)}{r(x)}r(x)}dx+\int_{(\R)_\infty^{q(\cdot)}}k_{q(\cdot)}^{r(x)}|f(x)|^{r(x)}dx\\
		&\geq k_{q(\cdot)}^{r^+}\int_\R |f(x)|^{p(x)}=k^{r^+}_{q(\cdot)}\rho_{p(\cdot)}(f)\geq\rho_{p(\cdot)}(k^{r^+}_{q(\cdot)} f).
	\end{align*}	
\end{proof}

\begin{proof}[Proof of Proposition \ref{propo norma conjugada}]
	By Hölder's inequality \ref{teo Hölder},
	\begin{equation*}
		\|f\|^\ast_{p(\cdot)}=\underset{\|g\|_{q(\cdot)}\leq1}{\sup}\|fg\|_{r(\cdot)}
		\lesssim\underset{\|g\|_{q(\cdot)}\leq1}{\sup}\|f\|_{p(\cdot)}\|g\|_{q(\cdot)}
		\leq\|f\|_{p(\cdot)}.
	\end{equation*}
	
	On the other hand, since $\||f|\|^\ast_{p(\cdot)}=\|f\|^\ast_{p(\cdot)}$, without loss of generality we can take $f\geq0$. Moreover, by homogeneity of the norm $\|\cdot\|^*_{p(\cdot)}$, is enough to consider $\|f\|^*_{p(\cdot)}=1$. 
	
	Let $B_k(0)$ be the ball centered at the origin with radius $k$. We define $E_k=B_k(0)\cap\left\{x\in\R:p(x)<k\right\}$, and $f_k(x)=\min\left\{f(x),k\right\} \chi_{E_k}(x)$. Therefore $f_k\nearrow f$, as $k \to \infty$,  $\|f_k\|^\ast_{p(\cdot)}\leq\|f\|^\ast_{p(\cdot)}=1$, and $\rho_{p(\cdot)}(f_k)<\infty$.
	 	
	\textbf{Case 1:} If $r(\cdot)<p(\cdot)<\infty$, result $(\R)^{q(\cdot)}_\infty=\emptyset$. Applying Lemma \ref{lema acotación de la modular para q finito} to $f_k$, we obtain $\rho_{p(\cdot)}(f_k)\leq1$. Therefore, by Fatou's lemma, $\rho_{p(\cdot)}(f)\leq1$, and consequently
	\begin{equation*}
	\|f\|_{p(\cdot)}\leq1=\|f\|^\ast_{p(\cdot)}.
	\end{equation*}
		
	\textbf{Case 2:} If $r(\cdot)\leq p(\cdot)<\infty$ and $r^+<\infty$. By Lemma \ref{lema acotación de la modular para p=r},\\ $\rho_{p(\cdot)}\left(k_{q(\cdot)}^{r^+} f_k\right)\leq 1$. Therefore, by Fatou's lemma,
	\begin{equation*}
		\|f\|_{p(\cdot)}\leq k_{q(\cdot)}^{-r^+}=k_{q(\cdot)}^{-r^+}\|f\|^\ast_{p(\cdot)}.
	\end{equation*} 
\end{proof}
Now we prove the Calderón-Zygmund decomposition.
\begin{proof}[Proof of Lemma \ref{lema descomposición de C-Z}]	
		The proof of \eqref{1 de CZ} and \eqref{2 de CZ} are an direct adaptation of the proofs of the Lemmas 3.9  and 3.13 of \cite{libro CU-FIO}. For \eqref{3 de CZ}, let $a_0 > 1$ be a constant to be determined later. For each $k\in\mathbb{Z}$ and $a>a_0$  let
		$$\Omega_k=\left\{x\in\R : M^{\mathcal{D}}_{{\beta(\cdot)},{r(\cdot)}}f(x)>a^k\right\}=\bigcup_{j\in {\mathbb{J}}_k} Q_j^k$$
		where $\{Q_j^k\}_{j\in{\mathbb{J}}_k}$ are the cubes given by $\eqref{1 de CZ}$ for $\lambda=a^k$. Therefore
			\begin{equation}\label{ecuacion 2 de C-Z}
				a^k<\|\chi_{Q_j^k}\|_{\beta(\cdot)}\frac{\|f\chi_{Q_j^k}\|_{r(\cdot)}}{\|\chi_{Q_j^k}\|_{r(\cdot)}}\leq D^2_{r(\cdot)} a^k.
			\end{equation}
		We define $E_j^k=Q_j^k\setminus\Omega_k$, then the sets $E_j^k$ are pairwise disjoint for each $k$. Moreover the sets $E_j^k$ are pairwise disjoints for all $k$ and $j$. In fact, if $k>l$, and   
		\begin{align*}
			x\in E_j^k\cap E_i^l=(Q_j^k\cap Q_i^l)\setminus \Omega_l.
		\end{align*}
		Then
		$a^k<M^{\mathcal{D}}_{{\beta(\cdot)},{r(\cdot)}}f(x)\leq a^l\leq a^k$ which is absurd. We prove that the family $\{Q_j^k\}_j$ is sparse, for this we most estimate
		$$\frac{|E_j^k|}{|Q_j^k|}=\frac{|Q_j^k\setminus\Omega_{1+k}|}{|Q_j^k|}=\frac{|Q_j^k\setminus(Q_j^k\cap\Omega_{k+1})|}{|Q_j^k|}=1-\frac{|Q_j^k\cap\Omega_{k+1}|}{|Q_j^k|}.$$	
		By the definition of $\Omega_k$, the maximality of $Q_j^k$, Hölder's inequality \ref{teo Hölder}, inequality \eqref{ecuacion 2 de C-Z} applied twice, Theorem \ref{teo suma de normas}, and Lemma \ref{norma de cubos},
		\begin{align*}
			|Q_j^k\cap\Omega_{k+1}|&=\sum_{Q_i^{k+1}\subset Q_j^k}\left|Q^{k+1}_i\right|\leq C \sum_{Q_i^{k+1}\subset Q_j^k}\left\|\chi_{Q^{k+1}_i}\right\|_{r(\cdot)}\left\|\chi_{Q^{k+1}_i}\right\|_{r^\prime(\cdot)}\\
			&\leq C \sum_{Q_i^{k+1}\subset Q_j^k}a^{-(k+1)}\left\|\chi_{Q^{k+1}_i}\right\|_{\beta(\cdot)}\left\|f\chi_{Q^{k+1}_i}\right\|_{r(\cdot)}\left\|\chi_{Q^{k+1}_i}\right\|_{r^\prime(\cdot)}\\
			&\leq C a^{-(k+1)} \left\|\chi_{Q^k_j}\right\|_{\beta(\cdot)}\sum_{Q_i^{k+1}\subset Q_j^k}\left\|f\chi_{Q_j^k}\chi_{Q^{k+1}_i}\right\|_{r(\cdot)}\left\|\chi_{Q^k_j}\chi_{Q^{k+1}_i}\right\|_{r^\prime(\cdot)}\\
			&\leq C a^{-(k+1)}\left\|\chi_{Q_j^k}\right\|_{\beta(\cdot)}\left\|f\chi_{Q_j^k}\right\|_{r(\cdot)}\left\|\chi_{Q_j^k}\right\|_{r^\prime(\cdot)}\\
			&\leq Ca^{-(k+1)}D^2_{r(\cdot)} a^k\|\chi_{Q_j^k}\|_{r(\cdot)}\|\chi_{Q_j^k}\|_{r^\prime(\cdot)}\leq \frac{C}{a}|Q_j^k|.
		\end{align*}
		Finally, taking $a_0=C$,
		\begin{equation*}
			\frac{|E_j^k|}{|Q_j^k|}=1-\frac{|Q_j^k\cap\Omega_k|}{|Q_j^k|}\geq 1- \frac{C}{a}>0.
		\end{equation*}
\end{proof}

\subsection*{Acknowledgement}
	We want to thank Gonzalo Ibañez-Firnkorn, for his generosity and knowledge given to this work.

\end{document}